\documentclass[11pt]{amsart}
\usepackage{mathtools, amssymb,amsfonts,amsthm,amscd,amsmath,amsgen,upref,enumerate,nicefrac,esint}

\usepackage[margin=1in]{geometry}

\theoremstyle{plain}
\newtheorem{cor}{Corollary}
\newtheorem{lem}[cor]{Lemma}
\newtheorem{prop}[cor]{Proposition}

\newtheorem{thm}[cor]{Theorem}
\newtheorem*{thm*}{Theorem}
\theoremstyle{definition}
\newtheorem{definition}[cor]{Definition}
\newtheorem{remark}[cor]{Remark}

\numberwithin{cor}{section}
\numberwithin{equation}{section}

\DeclareMathOperator{\C}{C}

\DeclareMathOperator{\Supp}{Supp}

\newcommand{\E}{\mathbb{E}}

\renewcommand{\d}{\delta}

\renewcommand{\and}{\quad\textrm{ and }\quad}

\renewcommand{\P}{\mathbb{P}}
\renewcommand{\a}{\alpha}

\renewcommand{\o}{\omega}
\renewcommand{\O}{\Omega}

\newcommand{\mcF}{\mathcal{F}}

\newcommand{\F}{\mathcal F}

\newcommand{\mcC}{\mathcal C}

\newcommand{\mcS}{\mathcal{S}}

\newcommand{\R}{\mathbb{R}}

\newcommand{\N}{\mathbb{N}}
\newcommand{\Z}{\mathbb{Z}}

\newcommand{\norm}[1]{\left\| #1 \right\|}

\newcommand{\ve}{\varepsilon}

\newcommand{\Prod}{\prod\limits}

\newcommand{\abs}[1]{\left|#1\right|}
\providecommand{\ud}[1]{\, \mathrm{d} #1}
\providecommand{\dx}{\ud{x}}
\providecommand{\dy}{\ud{y}}

\providecommand{\ds}{\ud{s}}
\providecommand{\dt}{\ud{t}}

\providecommand{\dd}{\ud}

\def\XXint#1#2#3{{\setbox0=\hbox{$#1{#2#3}{\int}$ }
\vcenter{\hbox{$#2#3$ }}\kern-.6\wd0}}

\usepackage{graphicx}
\usepackage{color}

\begin{document}

\title[Homogenization with divergence-free drift]{Stochastic homogenization with space-time ergodic divergence-free drift}

\author{Benjamin Fehrman}
\address{Mathematical Institute, University of Oxford, OX2 6GG Oxford, United Kingdom}
\email{Benjamin.Fehrman@maths.ox.ac.uk}

\begin{abstract}
We prove that diffusion equations with a space-time stationary and ergodic, divergence-free drift homogenize in law to a deterministic stochastic partial differential equation with Stratonovich transport noise.  In the absence of spatial ergodicity, the drift is only partially absorbed into the skew-symmetric part of the flux through the use of an appropriately defined stream matrix.  This leaves a time-dependent, spatially-homogenous transport which, for mildly decorrelating fields, converges to a Brownian noise with deterministic covariance in the homogenization limit.  The results apply to uniformly elliptic, stationary and ergodic environments in which the drift admits a suitably defined stationary and $L^2$-integrable stream matrix.
\end{abstract}

\maketitle

\section{Introduction}

In this paper, we consider the asymptotic behavior of solutions to the equation
\begin{equation}\label{intro_eq} \partial_t\rho^\ve = \nabla \cdot a^\ve \nabla\rho^\ve+\ve^{-1}b^\ve\cdot \nabla\rho^\ve+f \;\;\textrm{in}\;\;\R^d\times(0,T)\;\;\textrm{with}\;\;\rho^\ve(\cdot,0)=g,\end{equation}
for $f\in(L^2\cap \C)(\R^d\times(0,T])$, for $g\in(L^2\cap\C)(\R^d)$, and for the rescaled coefficients $a^\ve(x,t) = a(\nicefrac{x}{\ve},\nicefrac{t}{\ve^2})$ and $b^\ve(x,t)=b(\nicefrac{x}{\ve},\nicefrac{t}{\ve^2})$.  We will assume that the coefficients are stationary and ergodic:  there exist a probability space $(\O,\mathcal{F},\P)$, random variables $A\in L^\infty(\O;\R^{d\times d})$ and $B\in L^2(\O;\R^d)$, and a bi-measurable, measure-preserving, ergodic transformation group $\{\tau_{x,t}\colon\O\rightarrow\O\}_{(x,t)\in\R^{d+1}}$ such that, $\P$-a.e.,
\begin{equation}\label{intro_stationary}a(x,t) = A(\tau_{x,t}\o)\;\;\textrm{and}\;\;b(x,t)=B(\tau_{x,t}\o)\;\;\textrm{for every}\;\;\o\in\O.\end{equation}
We will assume that the diffusion matrix $A$ is uniformly elliptic:  there exist $\lambda\leq \Lambda\in(0,\infty)$ such that, $\P$-a.e.,
\begin{equation}\label{intro_elliptic} \lambda\abs{\xi}^2\leq \langle A\xi,\xi\rangle \;\;\textrm{and}\;\;\abs{A\xi}\leq \Lambda\abs{\xi}\;\;\textrm{for every}\;\;\xi\in\R^d,\end{equation}
and we will assume that the drift $B$ is $L^2$-integrable, mean-zero, and divergence-free:
\begin{equation}\label{intro_drift} B\in L^2(\O;\R^d),\;\E\left[B\right]=0\;\;\textrm{and, almost surely in the sense of distributions,}\;\;\nabla\cdot b = 0\;\;\textrm{on}\;\;\R^{d+1}.\end{equation}

We will now further remark on condition \eqref{intro_drift}.  In the case that the transformation group is spatially ergodic---that is, in the case that $\{\tau_{x,0}\colon\O\rightarrow\O\}_{x\in\R^d}$ is ergodic, it is known in $d\geq 3$ and for sufficiently mixing fields that \eqref{intro_drift} implies the existence of a stationary, skew-symmetric stream matrix $S\in L^2(\O;\R^{d\times d})$ that almost-surely satisfies, in the sense of distributions,
\begin{equation}\label{intro_stream_1}\nabla\cdot s = b\;\;\textrm{on}\;\;\R^{d+1}\;\;\textrm{for}\;\;s(x,t) = S(\tau_{x,t}\o).\end{equation}
In the absence of spatial ergodicity equation \eqref{intro_stream_1} is not solvable, in general, and must be replaced by the normalized equation
\begin{equation}\label{intro_smat}\nabla \cdot s(x,t) = b(x,t) - \underline{b}(t)\;\;\textrm{on}\;\;\R^{d+1}\;\;\textrm{for}\;\;\underline{b}(t)=\E\left[B|\mathcal{F}_{\R^d}\right](\tau_{0,t}\o),\end{equation}
for $\mathcal{F}_{\R^d}\subseteq\F$ the $\sigma$-algebra of subsets left invariant by spatial shifts of the environment:
\begin{equation}\label{intro_spatial} \mathcal{F}_{\R^d} = \{A\in\mathcal{F}\colon \P((\tau_{x,0}A) \triangle A)=0\;\;\textrm{for every}\;\;x\in\R^d\}.\end{equation}
See Proposition~\ref{prop_normalize} and Remark~\ref{rem_normalize} for a further discussion on the necessity of this normalization.

In order to partially absorb the drift into the skew-symmetric part of the flux, we will assume that there exists a stationary, $L^2$-integrable, skew-symmetric stream matrix in the sense of \eqref{intro_spatial}:
\begin{equation}\label{intro_stream} \textrm{There exists a skew-symmetric $S\in L^2(\O;\R^{d\times d})$ that almost surely satisfies \eqref{intro_spatial},}\end{equation}
which is always true if $d\geq 3$ and the spatial correlations of $B$ decay faster than a square.  See Appendix~\ref{sec_stream} for full details.  It then follows using the skew-symmetry of $S$ and \eqref{intro_smat} that \eqref{intro_eq} can be rewritten in the form
\begin{equation}\label{intro_neq} \partial_t\rho^\ve = \nabla \cdot (a^\ve+s^\ve) \nabla\rho^\ve+\ve^{-1}\underline{b}^\ve\cdot \nabla\rho^\ve +f \;\;\textrm{in}\;\;\R^d\times (0,T)\;\;\textrm{with}\;\;\rho^\ve(\cdot,0)=g,\end{equation}
for $s^\ve(x,t) = s(\nicefrac{x}{\ve},\nicefrac{t}{\ve^2})$ and $\underline{b}^\ve(t) = \underline{b}(\nicefrac{t}{\ve^2})$.  It remains to understand the effect of the spatially homogenous transport defined by the rescaling of $\underline{b}$.

A spatially homogenous transport can neither be absorbed into the diffusion nor be averaged by the diffusive properties of the environment, and will persist in the homogenization limit.  By considering the path $w^\ve$ defined by $w^\ve(t) = \ve^{-1}\int_0^t\underline{b}(\nicefrac{s}{\ve^2})\ds$, it holds under general assumptions on $\underline{b}$ that there exists $\Sigma\in\R^{d\times d}$ such that, as $\ve\rightarrow 0$, for every $T\in(0,\infty)$,
\begin{equation}\label{intro_weak_con}  w^\ve\rightarrow \Sigma B\;\textrm{in law on}\;\C([0,T];\R^d),\end{equation}
for $B$ a standard $d$-dimensional Brownian motion.  See Appendix~\ref{sec_path} for a further discussion of this point.

We are now prepared to state the main result.  We consider the steady assumption
\begin{equation}\label{steady}  \textrm{Assume \eqref{intro_stationary}, \eqref{intro_elliptic}, \eqref{intro_drift}, \eqref{intro_stream}, and \eqref{intro_weak_con},}\end{equation}
and prove that, under this assumption, the solutions of \eqref{intro_neq} converge in law to the solution of a deterministic stochastic PDE with Stratonovich transport noise.

\begin{thm}[cf.\ Theorem~\ref{thm_final} below] \label{thm_main}  Assume \eqref{steady}.  Then there exists a uniformly elliptic $\overline{a}\in \R^{d\times d}$ such that, as $\ve\rightarrow 0$, the solutions $\rho^\ve$ of \eqref{intro_neq} converge in law on $L^2([0,T];H^1(\R^d))\cap \C([0,T];L^2(\R^d))$ to the solution $\overline{\rho}$ of the stochastic PDE with Stratonovich transport noise
\begin{equation}\label{intro_spde}\partial_t\overline{\rho} = \nabla\cdot\overline{a}\nabla\overline{\rho}+\nabla\overline{\rho}\circ\Sigma \dd B_t+f \;\;\textrm{in}\;\;\R^d\times(0,\infty)\;\;\textrm{with}\;\;\overline{\rho}(\cdot,0)=g,\end{equation}
for $B$ a standard $d$-dimensional Brownian motion.  \end{thm}

We will now make three remarks on the scope of the results before giving an overview of the proof and literature.
The first concerns the relation to the periodic case.   If $\underline{b}$ is $L$-periodic, then due to the fact that $\underline{b}$ has zero mean, for every $t\in(0,\infty)$,
\[\abs{w^\ve(t)}\leq \ve^{-1}\int_0^{\ve^2L}\abs{\underline{b}(\nicefrac{s}{\ve^2})\ds} \leq \ve\int_0^L\abs{b(s)}\ds.  \]
We therefore conclude using the integrability of $\underline{b}$ that, as $\ve\rightarrow 0$, we have that $w^\ve\rightarrow 0$ uniformly in $\C([0,\infty);\R^d)$.  So, in the periodic case, Theorem~\ref{thm_main} recovers the well-known homogenization result for periodic coefficients with $\Sigma = 0$.  The appearance of the Brownian transport term is fundamentally stochastic, and requires the correlations of $\underline{b}$ to decay sufficiently quickly.  See Appendix~\ref{sec_path} for some further details.

The second concerns the stream matrix.  To illustrate the role of the stream matrix $S$ it is useful to consider two motivating examples.  The first is the case of a spatially homogenous drift, such as when $b(t)=x'(t)$ for $x(t)$ a smooth interpolation of a simple random walk.  Donsker's theorem proves that, as $\ve\rightarrow 0$,
\[w^\ve(t) = \ve^{-1}\int_0^t b(\nicefrac{t}{\ve^2})\dt = \ve \int_0^{\nicefrac{t}{\ve^2}}b'(t)\dt = \ve x(\nicefrac{t}{\ve^2}),\]
converges in law to a Brownian motion, and the solutions $\rho^\ve$ of 
\begin{equation}\label{no_hom}\partial_t\rho^\ve = \Delta\rho^\ve+\ve^{-1}b^\ve \cdot \nabla\rho^\ve+f,\end{equation}
converge in law to the solution $\overline{\rho}$ of the stochastic PDE with Stratonovich noise
\[\partial_t\overline{\rho} = \Delta\overline{\rho}+\nabla\overline{\rho}\circ \dd B_t+f.\]
In this case, the stream matrix is zero and does nothing to absorb the drift into the skew-symmetric part of the flux.  On the other hand, when the environment is spatially ergodic or when the drift is time-independent, we have using the mean-zero condition that
\begin{equation}\label{intro_conditional_exp}\underline{b} = \E\left[\underline{B}|\mathcal{F}_{\R^d}\right]=\E[\underline{B}]=0.\end{equation}
Hence, we have that $w^\ve=0$ for every $\ve\in(0,1)$ and therefore that $\Sigma = 0$, which recovers the known results of Kozma, T\'oth, and the author \cite{Feh2022, KozTot2017, Tot2018} on the homogenization of \eqref{intro_neq} in spatially ergodic environments.  The role of the stream matrix is therefore to split the drift into two parts, each satisfying one of the above two cases.  That is, a part with zero conditional expectation \eqref{intro_conditional_exp} that can be absorbed into the flux, and a part that acts as a spatially homogenous transport which persists in the homogenization limit and converges to a Brownian transport noise.

And finally, we observe that under the stronger assumption that, as $\ve\rightarrow 0$, the paths $w^\ve$ converge $\P$-a.e.\ in $\C([0,T];\R^d)$ or, for example, that the enhancements of $w^\ve$ converge in the rough path topology to some $\P$-a.e.\ $\alpha$-H\"older continuous rough path $z$, a much simplified version of the methods in this paper prove that the solutions of \eqref{intro_neq} converge $\P$-a.e.\ to the solution of the rough PDE
\begin{equation}\label{intro_rough}\partial_t\overline{\rho} = \nabla\cdot\overline{a}\nabla\overline{\rho}+\nabla\overline{\rho}\circ\dd z+f \;\;\textrm{in}\;\;\R^d\times(0,\infty)\;\;\textrm{with}\;\;\overline{\rho}(\cdot,0)=g,\end{equation}
understood in the sense that $\overline{\rho}$ is a solution of \eqref{intro_rough} if and only if $\tilde{\overline{\rho}}(x,t) = \overline{\rho}(x+z(t),t)$ solves the equation
\[\partial_t\tilde{\overline{\rho}} = \nabla\cdot\overline{a}\nabla\tilde{\overline{\rho}}+\tilde{f} \;\;\textrm{in}\;\;\R^d\times(0,\infty)\;\;\textrm{with}\;\;\overline{\rho}(\cdot,0)=g,\]
for $\tilde{f}(x,t) = f(x+z(t),t)$.\\

\subsection{Elements of the proof.}  The proof is based on the construction of sublinear homogenization correctors and the perturbed test function method.  However, in this case, there are several important differences to observe.  The standard two-scale expansion
\[\rho^\ve\simeq \overline{\rho}+\ve\phi^\ve_i\partial_i\overline{\rho},\]
suggests that, by equating powers in $\ve$, the homogenization correctors $\phi_i$ should satisfy the equation
\[\partial_t\phi_i = \nabla\cdot (a+s)(\nabla\phi_i+e_i)+\underline{b}\cdot (\nabla\phi_i+e_i).\]
However, as we observed above in the case \eqref{no_hom}, our expectation is that the effect of $\underline{b}$ will persist in the homogenization limit.  The correctors do not account for this effect, and for this reason we instead construct in Section~\ref{sec_correctors} homogenization correctors satisfying the equation
\[\partial_t\phi_i = \nabla\cdot (a+s)(\nabla\phi_i+e_i)+\underline{b}\cdot \nabla\phi_i,\]
which becomes, in the homogenization scaling for $\phi^\ve_i(x,t) = \phi_i(\nicefrac{x}{\ve},\nicefrac{t}{\ve^2})$,
\[\partial_t\phi^\ve_i = \nabla\cdot (a^\ve+s^\ve)(\nabla\phi_i+e_i)+\ve^{-1}\underline{b}\cdot \nabla\phi_i.\]
The singular forcing defined by $\ve^{-1}\underline{b}$ makes it intractable to obtain stable estimates for $\partial_t\phi^\ve_i$, which are necessary in order to prove the sublinearity of the correctors in strong topologies.  It is therefore convenient to work with the transported correctors $\tilde{\phi}^\ve_i(x,t) = \phi^\ve(x+w^\ve(t),t)$ for the path $w^\ve(t) = \ve^{-1}\int_0^t\underline{b}(\nicefrac{s}{\ve^2})\ds$, which solve the equations
\[\partial_t\tilde{\phi}^\ve_i = \nabla\cdot(\tilde{a}^\ve+\tilde{s}^\ve)(\nabla\tilde{\phi}^\ve_i+e_i),\]
with transported coefficients $\tilde{a}^\ve$ and $\tilde{s}^\ve$.  However, while the transported equation is well-behaved, obtaining it comes at the significant cost of destroying the stationarity of the environment.  It is no longer clear, for example, that transported, stationary quantities will satisfy any kind of large-scale averaging or ergodicity.  To address this, the essential observation is that if we knew a priori as $\ve\rightarrow 0$ that $\P$-a.e.\ the $\nabla\phi^\ve_i\rightharpoonup 0$ converge weakly in $L^2$ and if we knew that the paths $w^\ve$ were $\P$-a.e.\ converging strongly in $\C([0,T];\R^d)$, then it would remain the case that $\nabla\tilde{\phi}^\ve_i$ converges weakly to zero.  Assumption~\eqref{intro_weak_con} is much weaker than such $\P$-a.e.\ strong convergence.  But, when combined with the Skorokhod representation theorem, this intuition can be made precise and is enough to characterize the large-scale averages of transported, stationary quantities in probability in Lemma~\ref{lem_main} below.

We first identify the correctors by their stationary gradients in Proposition~\ref{prop_cor_1}, and based on Lemma~\ref{lem_main} we prove in Proposition~\ref{prop_sublinear} that they are sublinear in the sense that, for every $R\in (0,\infty)$, as $\ve\rightarrow 0$,
\[\ve\phi^\ve_i\rightarrow 0\;\;\textrm{strongly in $L^2(\mcC_R)$ in probability,}\]
for $\mcC_R = B_R\times[0,R^2]$.  The sublinearity plays an important role in Proposition \ref{prop_unique} to prove that the stationary gradients of the correctors are unique, which in turn plays an important role in defining in Proposition~\ref{prop_hom_mat} the effective coefficient
\[\overline{a}e_i = \E\left[(A+S)(\nabla\phi_i+e_i)\right],\]
and proving that it is uniformly elliptic.

The proof of homogenization is based on the transpose correctors introduced in Section~\ref{sec_transpose} and the perturbed test function method.  However, due to the low-integrability of the stream matrix $S$ and drift $B$, it is necessary to introduce several modifications to the standard technique.  In Theorem~\ref{thm_main}, we introduce stationary approximations of the true homogenization corrector to treat certain error terms and, once we have obtained tightness of all of the relevant quantities, we obtain the convergence in law using the Skorokhod representation theorem.

In Appendix~\ref{sec_stream} we provide general conditions that guarantee the existence of a stream matrix satisfying \eqref{intro_stream}.  This result is essentially well-known, going back to \cite{Koz1985} and \cite[Proposition~4, Proposition~5]{KozTot2017}, but we include the details here to explain the appearance of the conditional expectation.  In Appendix~\ref{sec_path}, we provide general conditions for $\underline{b}$ under which \eqref{intro_weak_con} holds true, based on results that can be found, for example, in Billingsley \cite[Chapter~4]{Bil1999}.

\subsection{Overview of the literature.}  In the case $A=I$, for a sufficiently regular drift, solutions of \eqref{intro_eq} are related to a rescaling of the stochastic differential equation
\[\dd X_t=\sqrt{2}\dd B_t+b\left(X_t,t\right)\dt,\]
which is the passive tracer model.  This is a simple approximation of transport in a turbulent, incompressible flow, and has been used to describe passive advected quantities such as temperature in applications to hydrology, meteorological sciences, and oceanography.  See, for instance, Csanady \cite{Csa1973}, Frish \cite{Fri1995}, and Monin and Yaglom \cite{MonYag2007,MonYag2007II} for more details.

The theory of homogenization for equations with periodic coefficients is contained in the references Bensoussan, Lions, and Papanicolaou \cite{BenLioPap2011} and Jikov, Kozlov, and Ole\u{\i}nik \cite{JikKozOle1994}.  The study of divergence form equations, and non-divergence form equations without drift, with random coefficients was initiated by Papanicolaou and Varadhan \cite{PapVar1981,PapVar1982}, Osada \cite{Osa1983}, and Kozlov \cite{Koz1985}.  The study of \eqref{intro_eq} with time-independent, divergence-free drift was initiated by \cite{Osa1983}, who considered the case that the drift admits a bounded stream matrix,  and Oelschl\"ager \cite{Oel1988} who established an invariance principle and homogenization for time-independent equations like \eqref{intro_eq} in probability on the whole space by assuming the existence of an $L^2$-integrable, $\C^2$-smooth stream matrix.  More recently, in the discrete setting, Kozma and T\'oth  \cite{KozTot2017} established an invariance principle in probability with respect to the environment for the analogous discrete random walk under the so-called $\mathcal{H}_{-1}$-condition, which is equivalent to the existence of a stationary, $L^2$-integrable stream matrix.  T\'oth \cite{Tot2018} then proved a quenched central limit theorem in this setting assuming the existence of an $L^{2+\d}$-integrable stream matrix, using an adaptation of Nash's moment bound.  The higher $L^{d\vee(2+\d)}$-integrability assumption was introduced in Avellaneda and Majda \cite{AveMaj1991} to prove the quenched homogenization of the parabolic version of \eqref{intro_eq} on the whole space.  The author \cite{Feh2022} proved that the existence of an $L^2$-integrable stream matrix was sufficient to prove the homogenization of \eqref{intro_neq} in the time-independent setting, and showed that under the higher integrability condition of \cite{AveMaj1991} the environment satisfies a large-scale H\"older regularity estimate and first-order Liouville principle.  Related problems under more restrictive integrability assumptions have been considered by Fannjiang and Komorowski \cite{FanKom1997}.

We also mention the works of Kipnis and Varadhan \cite{KipVar1986} and Sidoravicius and Sznitman \cite{SidSzn2004} who established a quenched invariance principle for random walks on supercritical percolation clusters in $d\geq 4$ and a quenched invariance principle for the random conductance model with uniformly elliptic, i.i.d.\ conductances in an arbitrary dimension, Andres, Barlow, Deuschel, and Hambly \cite{AndBarDeuHam2014} who establish a general invariance principle for the random conductance model under the assumption that the conductances almost surely percolate, and Deuschel and K\"osters \cite{DeuKos2008} who study random walks satisfying the bounded cycle condition.

Time-dependent problems have received somewhat less attention, and have essentially been restricted to the case that \eqref{intro_stream_1} is solvable.   Landim, Olla, and Yau \cite{LanOllYau1998} considered the case of a bounded stream matrix in the sense of \eqref{intro_stream_1}, Fannjiang and Komorowski \cite{FanKom1999,FanKom2002} treat the case, for example, of an $L^p$-integrable stream matrix in the sense of \eqref{intro_stream_1} for $p=d+2$, and Komorowski and Olla \cite{KomOll2001} treat $L^2$-integrable drifts satisfying a condition on their spatial energy spectrums that guarantees the existence of a stream matrix in the sense of \eqref{intro_stream_1}.  An essential difference between these works and the current work is the use of the stream matrix to exploit the time ergodicity and to observe the stochastic transport.  Komorowski and Olla \cite{KomOll2002} have provided a counterexample to the annealed homogenization of equations like \eqref{intro_eq} on the whole space for drifts that do not admit a square-integrable stream matrix.  We also mention that \cite{KomOll2002},  T\'oth and V\'alko \cite{TotVal2012}, and Cannizzaro, Haunschmid-Sibitz, Toninelli \cite{CanHauTon2021} have shown that a random walk or diffusion failing to satisfy the $\mathcal{H}_{-1}$-condition can exhibit superdiffusive behavior.

\subsection{Organization of the paper.}  The paper is organized as follows.  Section~\ref{sec_correctors} is split into four subsections.  Section~\ref{sec_fa} introduces the functional analytic techniques on $\O$ that will be used to construct the homogenization correctors, Section~\ref{sec_cor} characterizes the large-scale averages of stationary quantities transported by $w^\ve$ and identifies the homogenization correctors by their unique stationary gradients.  The transpose correctors are defined in Section~\ref{sec_transpose}, and the homogenized coefficient $\overline{a}$ is defined and shown to be uniformly elliptic in Section~\ref{sec_hom}.  We prove Theorem~\ref{thm_main} in Section~\ref{sec_homogenize}.  Appendix~\ref{sec_stream} shows that a stream matrix exists in the sense of \eqref{intro_smat} whenever $d\geq 3$ and whenever the correlations of $B$ decay faster than a square.  Finally, in Appendix~\ref{sec_path}, we provide general conditions that guarantee \eqref{intro_weak_con}.

\section{The homogenization correctors}\label{sec_correctors}

We will now introduce the probabilistic framework that will be used to construct the homogenization correctors and stationary stream matrix.  The tools developed in Section~\ref{sec_fa} will be used to lift the equation defining the homogenization correctors to the probability space.  In Section~\ref{sec_cor}, we prove a fundamental lemma that will be used to understand the large scale averages of transported quantities, and then prove the existence and uniqueness of sublinear homogenization correctors.  Section~\ref{sec_transpose} introduces the transpose correctors, and Section~\ref{sec_hom} defines the effective coefficient $\overline{a}$ and proves that $\overline{a}$ is uniformly elliptic.

\subsection{Functional analysis on $\O$.}\label{sec_fa}    We will first describe how we understand the corrector equation, and the equations defining the stream matrix, on the probability space $\O$.  Following the methods of \cite{Koz1985,PapVar1981,PapVar1982}, the transformation group $\{\tau_{x,t}\}_{(x,t)\in\R^{d+1}}$ defines so-called horizontal derivatives $\{D_i\}_{i\in\{0,1,\ldots d\}}$:  we define
\[\mathcal{D}(D_0) = \{f\in L^2(\O)\colon \lim_{h\rightarrow 0}h^{-1}(f(\tau_{0,h}\o)-f(\o))\;\textrm{exists strongly in}\;\;L^2(\O)\},\]
and, for every $i\in\{1,\ldots,d\}$,
\[\mathcal{D}(D_i)=\{f\in L^2(\O)\colon \lim_{h\rightarrow 0}h^{-1}(f(\tau_{he_i,0}\o)-f(\o))\;\textrm{exists strongly in}\;L^2(\O)\}.\]
The operators $D_i\colon \mathcal{D}(D_i)\rightarrow L^2(\O)$ defined by $D_if=\lim_{h\rightarrow 0}h^{-1}(f(\tau_{he_i}\o)-f(\o))$ are closed, densely defined operators on $L^2(\O)$, and we define the abstract Sobolev space $\mathcal{H}^1(\O)=\cap_{i=0}^d\mathcal{D}(D_i)$.  For $\phi\in \mathcal{H}^1(\O)$ we will write $D\phi=(D_1\phi,\ldots,D_d\phi)$ for the spatial gradient and $D_0\phi$ for the derivative in time.  We will write $\mathcal{H}^{-1}(\O)$ for the dual space of $\mathcal{H}^1(\O)$.

A natural class of test functions can be constructed by convolution:  for every $\psi\in\C^\infty_c(\R^{d+1})$ and $f\in L^\infty(\O)$ let $\psi_f\in L^\infty(\O)$ be the convolution
\[\psi_f(\o)=\int_{\R^{d+1}}f(\tau_{x,t}\o)\psi(x)\dx\dt,\]
and let $\mathcal{D}(\O)$ be the space of all such functions.  The dominated convergence theorem proves that every $\psi\in\mathcal{D}(\O)$ is infinitely differentiable with respect to the operators $D_i$, and that the space $\mathcal{D}(\O)$ is dense in $L^p(\O)$ for every $p\in[1,\infty)$.  The dual of $\mathcal{D}(\O)$ will be written $\mathcal{D}'(\O)$, and distributional inequalities are understood in $\mathcal{D}'(\O)$ in the sense that, for example, for $f\in L^1(\O)$,
\[D_if=0\;\;\textrm{if and only if}\;\;\E[fD_i\psi]=0\;\;\textrm{for every}\;\;\psi\in\mathcal{D}(\O).\]

The space of vector fields $L^2(\O;\R^d)$ admits the following Helmholtz decomposition.  The space of potential fields on $\O$ is defined by
\[L^2_{\textrm{pot}}(\O)=\overline{\left\{D\psi\in L^2(\O;\R^d)\colon \psi\in \mathcal{H}^1(\O)\right\}}^{L^2(\O;\R^d)},\]
which is the strong $L^2(\O;\R^d)$-closure of the space of $\mathcal{H}^1(\O)$-gradients.  The space of divergence-free fields is
\[L^2_{\textrm{sol}}(\O)=\{V\in L^2(\O;\R^d)\colon D\cdot V=0\}.\]
It follows that the space $L^2(\O;\R^d)$ then admits the orthogonal decomposition
\[L^2(\O;\R^d)=L^2_{\textrm{pot}}(\O)\oplus L^2_{\textrm{sol}}(\O),\]
which can be deduced, for example, from \cite[Proposition~2.5]{Feh2022}.  We are now prepared to explain the solvability condition \eqref{intro_smat}.

\begin{prop}\label{prop_normalize}  Let $\Phi\in L^2_{\textrm{pot}}(\O)$.  Then, it holds that
\[\E\left[\Phi |\mathcal{F}_{\R^d} \right] = 0,\]
for the sigma algebra \eqref{intro_spatial} of subsets that are invariant with respect to spatial shifts of the environment.  \end{prop}

\begin{proof}  Let $\phi\in \mathcal{H}^1(\O)$ and let $g\in L^\infty(\O)$ be $\F_{\R^d}$-measurable.  Then, for every $i\in\{1,\ldots,d\}$, using the strong convergence built into the definition of the domains of the $D_i$ and the fact that the transformation group preserves the measure,
\[\E[D_i\phi g ] = \E\left[\lim_{\abs{h}\rightarrow 0}(h^{-1}(\phi(\tau_{he_i,0}\o)-\phi(\o))g(\o)\right] =  \E\left[\lim_{\abs{h}\rightarrow 0}(h^{-1}(g(\tau_{-he_i,0}\o)-g(\o))\phi(\o)\right]=0,\]
where the final equality follows from the fact that $g$ is $\F_{\R^d}$-measurable.  The claim then follows from the density of the set $\{D\phi\colon \mathcal{H}^1(\O)\}$ in $L^2_{\textrm{pot}}(\O)$.   \end{proof}

\begin{remark}\label{rem_normalize}  In the case of the stream matrix $S$, it follows from Proposition~\ref{prop_normalize} that if $D\cdot S = B$ in $\O$ then we have that
\[0 = \E\left[D\cdot S|\mcF_{\R^d}\right] = \E\left[B|\mcF_{R^d}\right]\;\;\textrm{in}\;\;L^2(\O;\R^d).\]
\end{remark}

\subsection{The homogenization corrector}\label{sec_cor}  In this section, we first identify the homogenization correctors by their stationary gradients on the probability space in Proposition~\ref{prop_cor_1}.  Then, since we will primarily work with the transported correctors, we characterize in Lemma~\ref{lem_main} the large-scale averages of stationary quantities transported by the paths $w^\ve$.  This characterization will be used to prove the sublinearity of the transported correctors in Section~\ref{prop_sublinear}, and their uniqueness in Proposition~\ref{prop_unique}.  In particular, the proof of Proposition~\ref{prop_unique} is strongly motivated by \cite[Lemma~3.27]{Oel1988}, and extends \cite[Lemma~3.27]{Oel1988} and \cite[Proposition~2.4]{Feh2022} to the time-dependent and probabilistic settings.

\begin{prop}\label{prop_cor_1}  Assume \eqref{steady}.  Then for every $F\in L^2(\O;R^d)$ there exists $\Phi\in L^2_{\textrm{pot}}(\O)$ that satisfies the distributional equalities
\begin{equation}\label{hc_0}D_0\Phi_i = D_i\left( D\cdot (A+S)\Phi+\underline{B}\cdot \Phi + D\cdot F\right)\;\;\textrm{on $\O$ for every}\;\;i\in\{1,\ldots,d\}.\end{equation}
This implies $\P$-a.e.\ that there exists $\phi\in L^2_{\textrm{loc}}(\R;H^1_{\textrm{loc}}(\R^d))$ that satisfies
\[\nabla\phi(x,t) = \Phi(\tau_{x,t}\o)\;\;\textrm{and}\;\;\fint_{\mathcal{C}_1}\phi = 0,\]
for $\mathcal{C}_1 = B_1\times[0,1]$, and that satisfies distributionally
\[\partial_t\phi = \nabla\cdot(a+s)\nabla\phi+\underline{b}\cdot\nabla\phi+\nabla\cdot f,\]
for $f(x,t)=F(\tau_{x,t}\o)$.
\end{prop}

\begin{proof}  For every $n\in\N$ let $S_n\in L^\infty(\O;\R^{d\times d})$ be the skew-symmetric matrix defined by $(S_n)_{ij} = ((n\wedge S_ij)\vee -n)$, and let $\underline{B}_n\in L^\infty(\O;\R^d)$ be the vector defined by $(\underline{B}_n)_i = ((n\wedge \underline{B}_i\vee -n)$.  For every $\beta\in(0,1)$ and $n\in\N$ consider the equation
\begin{equation}\label{hc_1}\beta\phi+D_0\phi = \beta D_0(D_0\phi)+D\cdot(A+S_n)D\phi+\underline{B}_n\cdot D\phi+D\cdot F\;\;\textrm{in}\;\;\mathcal{H}^{1}(\O)\cap L^2(\O).\end{equation}
It follows from the fact that $\underline{B}_n$ is invariant with respect to spatial shifts of the environment and bounded that, for every $\phi\in\mathcal{H}^1(\O)\cap L^2(\O)$,
\[\E\left[\underline{B}_n\cdot D\phi\cdot\phi\right] = \frac{1}{2}\E\left[\underline{B}_n\cdot D(\phi)^2\right] = 0\;\;\textrm{and}\;\;\E\left[D_0\phi \phi\right] = \frac{1}{2}\E\left[D_0\phi^2\right]=0,\]
and it follows from the skew-symmetry and boundedness of $S_n$ that, for every $\phi\in\mathcal{H}^1(\O)$,
\[\E\left[S_nD\phi\cdot D\phi\right]=0.\]
It then follows from the Lax-Milgram theorem that, for every $\beta\in(0,1)$ and $n\in\N$, there exists a unique solution $\phi_{\beta,n}\in \mathcal{H}^1(\O)\cap L^2(\O)$ of \eqref{hc_1} that satisfies, using H\"older's inequality, Young's inequality, and \eqref{intro_elliptic} that
\begin{equation}\label{hc_2}\beta\E\left[\phi_{\beta,n}^2\right]+\beta\E\left[(D_0\phi_{n,\beta})^2\right]+\frac{\lambda}{2}\E\left[\abs{D\phi_{n,\beta}}^2\right]\leq \frac{1}{2\lambda}\E\left[\abs{F}^2\right].\end{equation}
Furthermore, since it holds that every $\phi\in \mathcal{H}^1(\O)\cap L^2(\O)$ satisfies the distributional equalities
\begin{equation}\label{hc_7}D_iD_j\phi = D_jD_i\phi\;\;\textrm{in}\;\;\mathcal{D}'(\O)\;\;\textrm{for every}\;\;i,j\in\{0,1,\ldots,d\},\end{equation}
we have for every $\beta\in(0,1)$ and $n\in\N$ that, distributionally for every $i\in\{1,\ldots,d\}$,
\begin{align}\label{hc_3}
& D_0D_i\phi_{n,\beta}  = D_iD_0\phi_{n,\beta}
\\ \nonumber & = D_i\left(-\beta\phi_{n,\beta}+\beta D_0(D_0\phi_{n,\beta})+D\cdot(A+S_n)D\phi_{n,\beta}+\underline{B}_n\cdot D\phi_{n,\beta}+D\cdot F\right)\;\;\textrm{in}\;\;D'(\O).
\end{align}
It then follows from \eqref{hc_2} that, after passing to a subsequence $n\rightarrow\infty$ and $\beta\rightarrow 0$, there exists $\Phi\in L^2_{\textrm{pot}}(\O)$ such that, along the subsequence $n\rightarrow\infty$ and $\beta\rightarrow 0$,
\begin{equation}\label{hc_4}\Phi_{n,\beta}\rightharpoonup \Phi\;\;\textrm{weakly in}\;\;L^2_{\textrm{pot}}(\O),\end{equation}
and, as $n\rightarrow\infty$,
\begin{equation}\label{hc_6}S_n\rightarrow S\;\;\textrm{and}\;\;\underline{B}_n\rightarrow\underline{B}\;\;\textrm{strongly in $L^2(\O;\R^{d\times d})$ and $L^2(\O;\R^d)$.}\end{equation}
It follows from \eqref{hc_2}, \eqref{hc_7}, \eqref{hc_4}, and \eqref{hc_6} that, after passing along along the subsequence $n\rightarrow\infty$ and $\beta\rightarrow 0$ in \eqref{hc_3}, distributionally for every $i,j\in\{1,\ldots,d\}$,
\begin{align}\label{hc_5}
& D_0\Phi_i  = D_i\left(D\cdot(A+S)\Phi+\underline{B}\cdot \Phi+D\cdot F\right)\;\;\textrm{and}\;\;D_i\Phi_j=D_j\Phi_i\;\;\textrm{in}\;\;D'(\O).
\end{align}
This completes the proof of \eqref{hc_0}, which proves that the random vector
\[(D\cdot(A+S)\Phi+\underline{B}\cdot \Phi+D\cdot F,\Phi)\;\;\textrm{is curl-free.}\]
The existence of $\phi$ then follows by integration, and the distributional equality is a consequence of \eqref{hc_5}.  This completes the proof.  \end{proof}

\begin{lem}\label{lem_main}  Assume \eqref{steady}, let $V\in L^1(\O)$, and let $v(x,t) = V(\tau_{x,t}\o)$.  For every $\ve\in(0,1)$ let $w^\ve(t) = \ve^{-1}\int_0^t\underline{b}(\nicefrac{s}{\ve^2})\ds$, let $v^\ve(x,t) = v(\nicefrac{x}{\ve},\nicefrac{t}{\ve^2})$, and let $\tilde{v}^\ve(x,t) = v^\ve(x+w^\ve(t),t)$.  Then, as $\ve\rightarrow 0$, for every $g\in \C_c(\R^d\times[0,\infty))$,
\[\int_0^\infty\int_{\R^d}\tilde{v}^\ve g\rightarrow \E[V]\int_0^\infty\int_{\R^d}g \;\;\textrm{in probability.}\]
\end{lem}

\begin{proof}  It follows from the ergodic theorem that, for every $n\in\N$,
\begin{equation}\label{lm_1} v^\ve \rightharpoonup \E[V]\;\;\textrm{weakly in $L^1(\mcC_n)$,}\end{equation}
for $\mcC_n = B_n\times[0,n^2]$, and it follows by stationarity that, for every $\ve\in(0,1)$ and $n\in\N$,
\begin{equation}\label{lm_2}\E\left[\int_{\mcC_n} \abs{v^\ve}\right] = \abs{\mcC_n}\E[\abs{V}].\end{equation}
For every $\ve\in(0,1)$ and $n\in\N$ consider the random variables
\[X^{\ve,n} = \left(\frac{v^\ve \dx\dt }{\norm{v^\ve}_{L^1(\mcC_n)}}\mathbf{1}_{\{v^\ve\neq 0\;\;\textrm{in}\;\;L^1(\mcC_n)\}},\norm{v^\ve}_{L^1(\mcC_n)},w^\ve\right),\]
taking values in the metric space $\overline{X}^n = B_1(\mathcal{M}(\mcC_n))\times \R\times\C([0,n^2];\R^d)$ equipped with the product metric $d_n$ induced by the separable and metrizable topology of weak-* convergence on the unit ball of the space of Radon measures $\mathcal{M}(\mcC_n)$ on $\mcC_n$, and the strong metric topologies of $\R$ and $\C([0,n^2];\R^d)$.  We then define the space
\[\overline{X}=\Prod_{n=1}^\infty \overline{X}^n,\]
equipped with the metric of coordinate-wise convergence $d$ defined by
\[d(x,y) = \sum_{n=1}^\infty 2^{-n}\frac{d_n(x_n,y_n)}{1+d_n(x_n,y_n)}\;\;\textrm{for every $x,y\in \overline{X}$.}\]
For every $\ve\in(0,1)$ we define the $\overline{X}$-valued random variable $X^\ve=(X^{n,\ve})_{n\in\N}$.  Since each of the $(X_n,d_n)$ are separable metric spaces, it follows that $(\overline{X},d)$ is a separable metric space.  It follows from the compactness of the unit ball of $\mathcal{M}(\mcC_n)$ in the weak-* topology, \eqref{lm_2}, assumption \eqref{intro_weak_con}, and Prokhorov's theorem (see, for example, \cite[Chapter~1, Theorem~5.1]{Bil1999}) that the laws of the $X^\ve$ are tight on $\overline{X}$.  Therefore, by Prokhorov's theorem, there exists a probability measure $\mu$ on $\overline{X}$ and a subsequence $\ve_k\rightarrow 0$ such that, as $k\rightarrow\infty$,
\[X^{\ve_k}\rightharpoonup \mu\;\;\textrm{in distribution on $\overline{X}$.}\]
The Skorokhod representation theorem (see, for example, \cite[Chapter~1, Theorem~6.7]{Bil1999}) then proves that there exists a probability space $(S,\mathcal{S},\mathbf{P})$ and $\overline{X}$-valued random variables $Y_k$ and $Y$ on $S$ such that, along a further subsequence $k\rightarrow\infty$, for every $k\in\N$,
\begin{equation}\label{lm_3} Y_k\sim X^{\ve_k}\;\;\textrm{and}\;\; Y\sim \mu\;\;\textrm{in distribution,}\end{equation}
and such that, $\mathbf{P}$-a.e.,
\begin{equation}\label{lm_04}\lim_{k\rightarrow\infty} Y_k = Y\;\;\textrm{in $\overline{X}$.}\end{equation}
We will write
\[Y_k = (Y_{n,k})_{n\in\N}=(m_{n,k},\alpha_{n,k},w_{n,k})_{n\in\N}\;\;\textrm{and}\;\;Y = (Y_n)_{n\in\N} = (m_n,\alpha_n,w_n)_{n\in\N}.\]
Since it follows from \eqref{lm_3} and the definition of $X^{\ve_k}$ that $\mathbf{P}$-a.e., for every $n\leq r\in \N$,
\[\alpha_{n,k}m_{n,k}=\alpha_{r,k}m_{r,k}\;\;\textrm{on}\;\;\mcC_n\;\;\textrm{and}\;\;w_{n,k}=w_{r,k}\;\;\textrm{on}\;\;[0,n^2],\]
and similarly for $Y$, for every $k\in\N$ we can $\mathbf{P}$-a.e.\ write the random vectors
\[(\alpha_{n,k}m_{n,k}, w_{n,k})_{n\in\N} = (M_k,W_k)\;\;\textrm{and}\;\;(\alpha_nm_n,w_n)_{n\in\N}=(M,W),\]
for $\mathbf{P}$-a.e.\ uniquely determined, locally finite Radon measures $M_k,M\in \mathcal{M}_{\textrm{loc}}(\R^d\times[0,\infty))$ and for uniquely determined paths $W_k,W\in \C([0,\infty);\R^d)$.  We will show $\mathbf{P}$-a.e.\ that $M = \E[V]\dx\dt$.  

Let $g\in \C_c(\R^d\times[0,\infty))$ have compact support and fix $n\in\N$ satisfying $\Supp(g)\subseteq\mcC_n$.  We first observe that the weak-* convergence of the $m_{n,k}$ and the strong convergence of the $\alpha_{n,k}$ prove that, $\mathbf{P}$-a.e.\ as $k\rightarrow\infty$,
\begin{equation}\label{lm_4}\alpha_{n,k} m_{n,k}\rightharpoonup \alpha_nm_n\;\;\textrm{weakly-* in $\mathcal{M}(\mcC_n)$.}\end{equation}
Then, since it follows from \eqref{lm_3} that, for every $k\in\N$,
\begin{equation}\label{lm_05}\alpha_{n,k} m_{n,k} = M_k \sim v^{\ve_k}\mathbf{1}_{\{v^{\ve_k}\neq 0\;\textrm{in}\;L^1(\mcC_n)\}}\;\;\textrm{in distribution on $\mathcal{M}(\mcC_n)$,}\end{equation}
we have from \eqref{lm_1}, \eqref{lm_4}, and the compact support of $g$ that, in probability,
\begin{align}\label{lm_5} \lim_{k\rightarrow\infty}\int_0^\infty\int_{\R^d} g M_k(\dx\dt) & =\lim_{k\rightarrow\infty}\int_{\mcC_n} g M_k(\dx\dt)
\\ \nonumber & = \int_{\mcC_n} g\alpha_n m_n(\dx\dt)
\\ \nonumber & =\E[V]\int_{\mcC_n}g,
\end{align}
which proves that $\mathbf{P}$-a.e.\ we have $\alpha_n m_n=\E[V]\dx\dt$.  It then follows by the definition of $M$, \eqref{lm_04}, and \eqref{lm_4} that $\mathbf{P}$-a.e.\ $M= \E[V]\dx\dt$.

In order to conclude, for every $k\in\N$ and $\ve\in(0,1)$, we define in the sense of distributions
\[\tilde{M}_k(x,t) = M_k(x+W_k(t),t),\]
and observe using \eqref{lm_3} that, for every $n\in \N$,
\[\tilde{M}_k\sim \tilde{v}^\ve\mathbf{1}_{\{v^{\ve_k}\neq 0\;\textrm{in}\;L^1(\mcC_n)\}}\dx\dt\;\;\textrm{in distribution on $\mathcal{M}(\mcC_n)$.}\]
Let $g\in \C_c(\R^d\times[0,\infty))$.  First fix $n\in\N$ such that $\Supp(g)\subseteq \mcC_n$, and then let $R_n\in[n,n+1,n+2,\ldots)$ be the smallest random integer larger than $n$ satisfying
\[\sup_{k\in\N}\left(\sup_{t\in[0,n^2]}\abs{w_k(t)}\right)\leq R_n,\]
which by \eqref{lm_04} is $\mathbf{P}$-a.e. finite.  We then observe using the definitions of $n$ and $R_n$ and the compact support of $g$ that, for every $k\in\N$,
\[\int_0^\infty\int_{\R^d} g(x,t)\tilde{M}_k(\dx\dt) = \int_0^\infty\int_{\R^d}g(x-W_k(t),t)M_k(\dx\dt)=\int_{\mcC_{\overline{R}_n}}g(x-W_k(t),t)M_k(\dx\dt),\]
for $\overline{R}_n = n+R_n$.  It follow from the uniform boundedness of the $(\alpha_{\overline{R}_n,k})_{k\in\N}$, which implies by definition the uniform boundedness of the $(M_k)_{k\in\N}$ in the total variation norm on $\mcC_{\overline{R}_n}$, and the uniform continuity of $g$ that $\mathbf{P}$-a.e.\ there exists a random $c\in(0,\infty)$ and a modulus of continuity $\a$ independent of $k$ such that, for every $k\in\N$,
\[\abs{\int_{\mcC_{\overline{R}_n}}\left(g(x-W_k(t),t)-g(x-W(t),t)\right)M_k(\dx\dt)}\leq c\a\left(\sup_{t\in[0,\overline{R}_n^2]}\abs{W_k(t)-W(t)}\right).\]
It then follows from \eqref{lm_04} that, $\mathbf{P}$-a.e.,
\[\limsup_{k\rightarrow\infty}\abs{\int_{\mcC_{\overline{R}_n}}\left(g(x-W_k(t),t)-g(x-W(t),t)\right)M_k(\dx\dt)}=0.\]
We therefore have from \eqref{lm_5}, a change of variables, and the definition of $\overline{R}_n$ that
\begin{align*}
\lim_{k\rightarrow\infty}\int_{\mcC_{\overline{R}_n}}g(x-W_k(t),t)M_k(\dx\dt) & =\lim_{k\rightarrow\infty}\int_{\mcC_{\overline{R}_n}}g(x-W(t),t)M_k(\dx\dt)
\\ & = \E[V]\int_{\mcC_{\overline{R}_n}}g(x-W(t),t)\dx\dt
\\ & = \E[V]\int_0^\infty\int_{\R^d}g(x,t)\dx\dt.
\end{align*}
It then follows from \eqref{lm_05} that, for every $g\in \C_c(\R^d\times[0,\infty))$,
\begin{equation}\label{lm_6}\lim_{k\rightarrow\infty}\int_0^\infty\int_{\R^d}\tilde{v}^{\ve_k}g\dx\dt =  \E[V]\int_0^\infty\int_{\R^d}g(x,t)\dx\dt\;\;\textrm{in probability.}\end{equation}
Since for every subsequence $\ve_k\rightarrow 0$ the above analysis proves that there exists a further subsequence $\ve_k\rightarrow 0$ such that \eqref{lm_6} holds for every $g\in\C_c(\R^d\times[0,\infty))$, we conclude that, for every $g\in \C_c(\R^d\times[0,\infty))$,
\[\lim_{\ve\rightarrow 0}\int_0^\infty\int_{\R^d}\tilde{v}^{\ve}g\dx\dt =  \E[V]\int_0^\infty\int_{\R^d}g(x,t)\dx\dt\;\;\textrm{in probability.}\]
This completes the proof.  \end{proof}

\begin{prop}\label{prop_sublinear}  Assume~\eqref{steady}, let $F\in L^2(\O;\R^d)$, let $\Phi$ and $\phi$ be as constructed in Proposition~\ref{prop_cor_1} corresponding to $F$, and for every $\ve\in(0,1)$ let $w^\ve = \ve^{-1}\int_0^t\underline{b}(\nicefrac{s}{\ve^2})\ds$, let $\phi^\ve(x,t) = \ve\phi(\nicefrac{x}{\ve},\nicefrac{t}{\ve^2})$, and let $\tilde{\phi}^\ve(x,t) = \phi^\ve(x+w^\ve(t),t)$.  Then, for every $R\in(0,\infty)$,
\[\lim_{\ve\rightarrow 0}\int_{\mcC_R}(\tilde{\phi}^\ve-\langle\tilde{\phi}^\ve\rangle_{\mcC_R})^2 = 0\;\;\textrm{in probability,}\]
for $\mathcal{C}_R = B_R\times [0,R^2]$ and for $\langle\tilde{\phi}^\ve\rangle_{\mcC_R}=\fint_{\mcC_R}\tilde{\phi}^\ve$.
\end{prop}

\begin{proof}  We first observe that, almost surely on $\R^{d+1}$,
\[\nabla\phi^\ve(x,t) = \Phi(\tau_{\nicefrac{x}{\ve},\nicefrac{t}{\ve^2}}\o)\;\;\textrm{and}\;\;\partial_t\phi^\ve = \nabla\cdot(a^\ve+s^\ve)\nabla\phi^\ve+\ve^{-1}\underline{b}(\nicefrac{t}{\ve})\cdot\nabla\phi^\ve+\nabla\cdot f^\ve,\]
for $a^\ve(x,t)=a(\nicefrac{x}{\ve},\nicefrac{t}{\ve^2})$ and similarly for $s^\ve$, $\underline{b}^\ve$, and $f^\ve$.  Let $w^\ve(t) = \ve^{-1}\int_0^t\underline{b}(\nicefrac{s}{\ve^2})\ds$ and observe that the function
\[\tilde{\phi}^\ve(x,t) = \phi(x+w^\ve(t),t)\;\;\textrm{on}\;\;\R^d\times[0,\infty),\]
is a distributional solution of the equation
\begin{equation}\label{hc_9}\partial_t\tilde{\phi}^\ve = \nabla\cdot(\tilde{a}^\ve+\tilde{s}^\ve)\nabla\tilde{\phi}^\ve+\nabla\cdot \tilde{f}^\ve\;\;\textrm{in}\;\;\R^d\times[0,\infty),\end{equation}
for $\tilde{a}^\ve(x,t) = a^\ve(x+w^\ve(t),t)$ and similarly for $\tilde{s}^\ve$ and $\tilde{f}^\ve$.  Finally, let $q^\ve$ denote the flux $q^\ve = (a^\ve+s^\ve)\nabla\phi^\ve+f^\ve$ and let $\tilde{q}^\ve(x,t) = q^\ve(x+w^\ve(t),t)$.  For every $R\in(0,\infty)$ and locally integrable $f\colon\R^d\times[0,\infty)\rightarrow\R$ let $\langle f \rangle_{\mcC_R}=\fint_{\mcC_R}f$ and, for every $R\in(0,\infty)$ and $t\in[0,R^2]$, we define the average on the time slice of $\mcC_R$ by $\langle f \rangle_{R,t} = \fint_{B_R}f(\cdot,t)$.  We will also consider a spatial smoothing of $\tilde{\phi}^\ve$:  for every $\d\in(0,1)$ let $\kappa^\d\colon\R^d\rightarrow\R$ be a standard symmetric spatial convolution kernel of scale $\d$ on $\R^d$ satisfying $\norm{\nabla\kappa^\d}\leq \nicefrac{c}{\d}$ for some $c\in(0,\infty)$ independent of $\d$, and for every $\ve\in(0,1)$ and $\delta\in(0,1)$ let $\tilde{\phi}^{\ve,\d}(x,t) = (\tilde{\phi}^\ve(\cdot,t)*\kappa^\d)(x)$.

Fix $R\in(0,\infty)$ and observe using the triangle inequality that, for every $\ve,\d\in(0,1)$,
\begin{align}\label{hc_20}
\int_{\mcC_R} (\tilde{\phi}^\ve-\langle\tilde{\phi}^\ve\rangle_{\mcC_R})^2 & \leq 2\int_{\mcC_R} ((\tilde{\phi}^\ve-\langle\tilde{\phi}^\ve\rangle_{\mcC_R})-(\tilde{\phi}^{\ve,\d}-\langle\tilde{\phi}^{\ve,\d}\rangle_{\mcC_R}))^2
\\ \nonumber & \quad +4\int_{\mcC_R} (\tilde{\phi}^{\ve,\d}-\langle\tilde{\phi}^{\ve,\d}\rangle_{R,t})^2+4\int_{\mcC_R} (\langle\tilde{\phi}^{\ve,\d}\rangle_{R,t}-\langle\tilde{\phi}^{\ve,\d}\rangle_{\mcC_R})^2.
\end{align}
For the first term on the righthand side of \eqref{hc_20}, it follows from $\delta\in(0,1)$, the definition of $\kappa^\d$, the definition of the convolution, and Jensen's inequality that
\[\int_{\mcC_R} ((\tilde{\phi}^\ve-\langle\tilde{\phi}^\ve\rangle_{\mcC_R})-(\tilde{\phi}^{\ve,\d}-\langle\tilde{\phi}^{\ve,\d}\rangle_{\mcC_R}))^2\leq \delta^2\int_0^\infty\int_{\R^d}\chi_{R}\abs{\nabla\tilde{\phi}^\ve}^2,\]
for any arbitrary smooth function $\chi_{R}\colon\R^d\times[0,\infty)\rightarrow [0,1]$ satisfying $\chi_{R}=1$ on $\mcC_{R+1}$ and $\chi_{R}=0$ on $(\R^d\times[0,\infty))\setminus\mcC_{R+2}$.

For the second term on the righthand side of \eqref{hc_20}, it follows from $\d\in(0,1)$, the definition of $\kappa^\d$, the definition of the convolution, and the Poincar\'e inequality applied to each individual time-slice of $\mcC_R$ that there exists $c\in(0,\infty)$ depending on $R$ such that
\begin{equation}\label{hc_22}\norm{(\tilde{\phi}^{\ve,\d}-\langle\tilde{\phi}^{\ve,\d}\rangle_{R,t})}^2_{L^2([0,R^2];H^1(B_R))}\leq c\int_0^\infty\int_{\R^d}\chi_{R}\abs{\nabla\tilde{\phi}^{\ve,\d}}^2.\end{equation}
We also have from \eqref{hc_9} that distributionally
\begin{equation}\label{hc_023}\partial_t(\tilde{\phi}^{\ve,\d}-\langle\tilde{\phi}^{\ve,\d}\rangle_{R,t})(x,t) = -\int_{\R^d}\tilde{q}^\ve(y,t)\cdot\nabla\kappa^\d(x-y)\dy,\end{equation}
and therefore, we have using the definition of $\kappa^\d$ above and $\delta\in(0,1)$ that, for some $c\in(0,\infty)$ independent of $\d\in(0,1)$,
\begin{equation}\label{hc_23}\norm{\partial_t(\tilde{\phi}^{\ve,\d}-\langle\tilde{\phi}^{\ve,\d}\rangle_{R,t})}_{L^1([0,R^2];L^1(B_R))}\leq c\delta^{-1}\int_0^\infty\int_{\R^d}\chi_{R}\abs{\tilde{q}^\ve}.\end{equation}

For the final term on the righthand side of \eqref{hc_20}, we first observe using \eqref{hc_9} that distributionally
\[\partial_t(\langle\tilde{\phi}^{\ve,\d}\rangle_{R,t}-\langle\tilde{\phi}^{\ve,\d}\rangle_{\mcC_R}) = \partial_t\langle\tilde{\phi}^{\ve,\d}\rangle_{R,t} = -\int_{B_R}\int_{\R^d}\tilde{q}^\ve(y,t)\cdot \nabla \kappa^\delta(y-x)\dy\dx,\]
and, therefore, for some $c\in(0,\infty)$ independent of $\d\in(0,1)$,
\[\abs{\partial_t(\langle\tilde{\phi}^{\ve,\d}\rangle_{R,t}-\langle\tilde{\phi}^{\ve,\d}\rangle_{\mcC_R})}\leq c\delta^{-1}\int_{\R^d}\chi_{R}(\cdot,t)\abs{\tilde{q}^\ve(\cdot,t)}.\]
We then observe that, by the fundamental theorem of calculus or one-dimensional Poincar\'e inequality, for every $s\leq t\in[0,R^2]$, for some $c\in(0,\infty)$ independent of $\delta$,
\begin{align}\label{hc_24}
& \norm{\partial_t(\langle\tilde{\phi}^{\ve,\d}\rangle_{R,t}-\langle\tilde{\phi}^{\ve,\d}\rangle_{\mcC_R})}_{L^1([0,R^2])}+\norm{(\langle\tilde{\phi}^{\ve,\d}\rangle_{R,t}-\langle\tilde{\phi}^{\ve,\d}\rangle_{\mcC_R})}_{L^\infty([0,R^2])}
\\ \nonumber & \leq c\delta^{-1}\int_0^\infty\int_{\R^d}\chi_{R}\abs{\tilde{q}^\ve}.
\end{align}

Let $\{\ve_k\}_{k\in\N}$ be an arbitrary sequence satisfying $\ve_k\rightarrow 0$ as $k\rightarrow\infty$.  Then by Lemma~\ref{lem_main} and $\chi_{R}\in\C_c(\R^d\times[0,\infty))$ there exists a further subsequence still denoted $\ve_k\rightarrow 0$ as $k\rightarrow\infty$ such that, $\P$-a.e.,
\begin{equation}\label{hc_50}\lim_{k\rightarrow\infty}= \int_0^\infty\int_{\R^d}\chi_{R}\abs{\tilde{q}^{\ve_k}} = \E[(A+S)\Phi+F]\int_{\R^d}\chi_{R},\end{equation}
and
\begin{equation}\label{hc_51}\lim_{k\rightarrow\infty}\int_0^\infty\int_{\R^d}\chi_{R}\abs{\nabla\tilde{\phi}^{\ve,\d}}^2 = E[\abs{\Phi}^2] \int_{\R^d}\chi_{R}.\end{equation}
From \eqref{hc_22}, \eqref{hc_23}, \eqref{hc_50}, and \eqref{hc_51} it follows $\P$-a.e.\ that the functions
\[(\tilde{\phi}^{\ve_k,\d}-\langle\tilde{\phi}^{\ve_k,\d}\rangle_{R,t})_{k\in\N}\;\;\textrm{are uniformly bounded in}\;\;L^2([0,R^2];H^1(B_R))\cap W^{1,1}([0,R^2];L^1(B_R)),\]
from which it follows using the compact embedding of $H^1(B_R)$ into $L^2(B_R)$ and the continuous embedding of $L^2(B_R)$ into $L^1(B_R)$, and the Aubin-Lions-Simon lemma \cite{Aub1963,Lio1969,Sim1987} that the functions
\[(\tilde{\phi}^{\ve_k,\d}-\langle\tilde{\phi}^{\ve_k,\d}\rangle_{R,t})_{k\in\N}\;\;\textrm{are relatively strongly compact in}\;\;L^2([0,R^2];L^2(B_R)).\]
Similarly, it follows $\P$-a.e.\ from \eqref{hc_24} that the functions
\[(\langle\tilde{\phi}^{\ve_k,\d}\rangle_{R,t}-\langle\tilde{\phi}^{\ve_k,\d}\rangle_{\mcC_R})_{k\in\N}\;\;\textrm{are uniformly bounded in}\;\;L^\infty([0,R^2])\cap W^{1,1}([0,R^2]),\]
from which it follows from the compact embedding of $W^{1,1}([0,R^2])$ into $L^1([0,R^2])$ and the $L^\infty$-boundedness that the functions
\[(\langle\tilde{\phi}^{\ve_k,\d}\rangle_{R,t}-\langle\tilde{\phi}^{\ve_k,\d}\rangle_{\mcC_R})_{k\in\N}\;\;\textrm{are relatively compact in}\;\;L^2([0,R^2]).\]
To characterize these strong limits, we now observe using Lemma~\ref{lem_main} applied component-wise and the symmetry of $\kappa^\d$ that, in probability, for every $\d\in(0,1)$ and $g\in\C_c(\R^d\times[0,\infty))$,
\begin{equation}\label{hc_52}\lim_{k\rightarrow\infty}\int_0^\infty\int_{\R^d}\nabla \tilde{\phi}^{\ve_k,\d}(x)\cdot g = \E[\Phi]\int_0^\infty\int_{\R^d}(g*\kappa^\d) = 0,\end{equation}
and it follows from Lemma~\ref{lem_main}, \eqref{hc_023}, and the symmetry of $\kappa^\d$ that, for every $\d\in(0,1)$ and $g\in\C_c(\R^d\times[0,\infty))$,
\begin{equation}\label{hc_53}\lim_{k\rightarrow\infty}\int_0^\infty\int_{\R^d}\partial_t\tilde{\phi}^{\ve_k,\d}(x)\cdot g = \E[(A+S)\Phi+F]\cdot \int_0^\infty\int_{\R^d}\nabla (g*\kappa^\d) = 0.\end{equation}
In combination, \eqref{hc_52} and \eqref{hc_53} prove that any potential strong limit of the functions $(\tilde{\phi}^{\ve,\d}-\langle\tilde{\phi}^{\ve,\d}\rangle_{R,t})_{k\in\N}$ in $L^2(\mcC_R)$ has vanishing distributional gradient and derivative in time, and must therefore be constant.  However, since each of the $(\tilde{\phi}^{\ve,\d}-\langle\tilde{\phi}^{\ve,\d}\rangle_{R,t})_{k\in\N}$ has average zero, this means that every potential strong limit is zero.   The same reasoning applies to the functions $(\langle\tilde{\phi}^{\ve,\d}\rangle_{R,t}-\langle\tilde{\phi}^{\ve,\d}\rangle_{\mcC_R})_{k\in\N}$.  We therefore have $\P$-a.e.\ that, along the subsequence,
\begin{equation}\label{hc_31}\lim_{k\rightarrow\infty} (\tilde{\phi}^{\ve_k,\d}-\langle\tilde{\phi}^{\ve_k,\d}\rangle_{R,t}))=\lim_{k\rightarrow\infty}(\langle\tilde{\phi}^{\ve_k,\d}\rangle_{R,t}-\langle\tilde{\phi}^{\ve_k,\d}\rangle_{\mcC_R})=0\;\;\textrm{strongly in}\;\;L^2(\mcC_R).\end{equation}
Returning to \eqref{hc_20}, we have from \eqref{hc_51} and \eqref{hc_31} that, along a further subsequence,
\[\limsup_{k\rightarrow\infty}\int_{\mcC_R} (\tilde{\phi}^{\ve_k}-\langle\tilde{\phi}^{\ve_k}\rangle_{\mcC_R})^2\leq \E[\abs{\Phi}^2]\delta^2,\]
which, after taking $\delta\rightarrow 0$, proves $\P$-a.e.\ that, along the subsequence,
\begin{equation}\label{hc_55}\lim_{k\rightarrow\infty}\int_{\mcC_R} (\tilde{\phi}^{\ve_k}-\langle\tilde{\phi}^{\ve_k}\rangle_{\mcC_R})^2=0.\end{equation}
Since we have shown that, for any subsequence $\ve_k\rightarrow 0$ as $k\rightarrow\infty$ there exits a further subsequence $\ve_k\rightarrow 0$ as $k\rightarrow\infty$ satisfying \eqref{hc_55}, it follows that, for every $R\in(0,\infty)$,
\[\lim_{\ve\rightarrow 0}\int_{\mcC_R} (\tilde{\phi}^\ve-\langle\tilde{\phi}^\ve\rangle_{\mcC_R})^2=0\;\;\textrm{in probability,}\]
which completes the proof.  \end{proof}

\begin{prop}\label{prop_unique}  Assume~\eqref{steady}.  Then for every $F\in L^2(\O;\R^d)$ the gradient $\Phi\in L^2_{\textrm{pot}}(\O)$ constructed in Proposition~\ref{prop_cor_1} satisfies the energy equality
\begin{equation}\label{energy}\E[A\Phi\cdot \Phi ] = -\E[F\cdot \Phi].\end{equation}
In particular, this implies that $\Phi$ is unique.  \end{prop}

\begin{proof}  Let $\phi$ be $\P$-a.e.\ defined as the unique $L^2_{\textrm{loc}}(\R;H^1_{\textrm{loc}}(\R^d))$ function satisfying
\[\nabla\phi(x,t) = \Phi(\tau_{x,t}\o)\;\;\textrm{and}\;\;\int_{\mcC_1}\phi = 0,\]
and distributionally that
\[\partial_t\phi = \nabla\cdot(a+s)\nabla\phi+\underline{b}\cdot\nabla\Phi+\nabla\cdot f.\]
For every $\ve\in(0,1)$ let $w^\ve(t) = \ve^{-1}\int_0^t\underline{b}(\nicefrac{s}{\ve^2})\ds$, let $\phi^\ve(x,t) = \ve\phi(\nicefrac{x}{\ve},\nicefrac{t}{\ve^2})$, and let $\tilde{\phi}^\ve = \phi(x+w^\ve(t),t)$.  It then holds distributionally that
\[\partial_t\tilde{\phi}^\ve = \nabla\cdot(\tilde{a}^\ve+\tilde{s}^\ve)\nabla\tilde{\phi}^\ve+\nabla\cdot\tilde{f}^\ve.\]
Since it follows by an approximation argument using the density of smooth functions in the space $L^2_{\textrm{loc}}([0,T];H^1_{\textrm{loc}}(\R^d))$ that, for every smooth function $g\in\C^\infty_c(\R^{d+1})$,
\[\int_{\R^{d+1}} \tilde{s}^\ve\nabla\tilde{\phi}^\ve\cdot \nabla g = \int_{\R^d}(\tilde{b}^\ve-\underline{b}^\ve)\nabla\tilde{\phi}^\ve g,\]
we have that $\tilde{\phi}^\ve$ $\P$-a.e.\ distributionally solves
\begin{equation}\label{hc_70}\partial_t\tilde{\phi}^\ve = \nabla\cdot\tilde{a}^\ve\nabla\tilde{\phi}^\ve+\ve^{-1}(\tilde{b}^\ve-\underline{b}^\ve)\nabla\tilde{\phi}^\ve+\nabla\cdot\tilde{f}^\ve.\end{equation}
Let $\chi\in\C^\infty_c(\R^d\times(0,\infty))$ be an arbitrary smooth function satisfying $\int_0^\infty\int_{\R^d}\chi = 1$ and $\Supp(\chi)\subseteq\mcC_2$, for every $K\in\N$ let $\psi_K\colon\R\rightarrow [-K,K]$ be defined by $\psi_K(x) = K$ if $x\geq K$, $\psi_K(x)=-K$ if $x\leq -K$, and  $\psi_K(x)=x$ if $\abs{x}\leq K$, and for every $K\in\N$ let $\Psi_K\colon\R\rightarrow\R$ be the unique function satisfying $\Psi_K(0)=0$ and $\Psi_K'(x)=\psi_K(x)$.

After first introducing a spatial convolution of $\tilde{\phi}^\ve$, it follows by an approximation argument that, for every $\ve, \d\in(0,1)$ and $K\in\N$, the function $\chi(x,t)\psi_K(\tilde{\phi}^\ve(x,t)-\langle\tilde{\phi}^\ve\rangle_{\mcC_2})$ is an admissible test function for \eqref{hc_70}.  Then, using the properties of $\chi$, we have for every $\ve\in(0,1)$ and $K\in\N$ that
\begin{align}\label{hc_90} -\int_{\mcC_2}\Psi_K(\tilde{\phi}^\ve(x,t)-\langle\tilde{\phi}^\ve\rangle_{\mcC_2})\partial_t\chi & =-\int_{\mcC_2}\psi'_K(\tilde{\phi}^\ve(x,t)-\langle\tilde{\phi}^\ve\rangle_{\mcC_2})\left(\tilde{a}^\ve\nabla\tilde{\phi}^\ve\cdot\nabla\tilde{\phi}^\ve+\tilde{f}^\ve\cdot\nabla\tilde{\phi}^\ve\right)\chi
\\ \nonumber & \quad +\int_{\mcC_2}\ve^{-1}(\tilde{b}^\ve-\underline{b}^\ve)\cdot\nabla\tilde{\phi}^\ve\psi_K(\tilde{\phi}^\ve(x,t)-\langle\tilde{\phi}^\ve\rangle_{\mcC_2})\chi.
\end{align}
For the final term of \eqref{hc_90}, we have using the distributional equality
\[\nabla\tilde{\phi}^\ve\psi_K(\tilde{\phi}^\ve(x,t)-\langle\tilde{\phi}^\ve\rangle_{\mcC_2}) = \nabla\Psi_K(\tilde{\phi}^\ve(x,t)-\langle\tilde{\phi}^\ve\rangle_{\mcC_2}),\]
that, after repeating the approximation argument above,
\begin{equation}\label{hc_91}\int_{\mcC_2}\ve^{-1}(\tilde{b}^\ve-\underline{b}^\ve)\cdot\nabla\tilde{\phi}^\ve\psi_K(\tilde{\phi}^\ve(x,t)-\langle\tilde{\phi}^\ve\rangle_{\mcC_2})\chi = \int_{\mcC_2}\tilde{s}^\ve\nabla\tilde{\phi}^\ve\psi_K(\tilde{\phi}^\ve(x,t)-\langle\tilde{\phi}^\ve\rangle_{\mcC_2})\cdot \nabla\chi.\end{equation}
Our aim is to apply Lemma~\ref{lem_main}, but the corrector itself is not stationary.  We therefore consider the set
\[A_{K,\ve} = \mcC_2\cap \{(x,t)\colon(\tilde{\phi}^\ve(x,t)-\langle\tilde{\phi}^\ve\rangle_{\mcC_2})\geq K\},\]
and observe by Chebyshev's inequality that in measure
\[\abs{A_{K,\ve}}\leq \frac{1}{K^2}\int_{\mcC_2} (\tilde{\phi}^\ve(x,t)-\langle\tilde{\phi}^\ve\rangle_{\mcC_2})^2.\]
For the first term on the righthand side of \eqref{hc_90} we have that
\begin{align}\label{hc_92} \int_{\mcC_2}\psi'_K(\tilde{\phi}^\ve(x,t)-\langle\tilde{\phi}^\ve\rangle_{\mcC_2})\left(\tilde{a}^\ve\nabla\tilde{\phi}^\ve\cdot\nabla\tilde{\phi}^\ve+\tilde{f}^\ve\cdot\nabla\tilde{\phi}^\ve\right)\chi & = \int_{\mcC_2}\left(\tilde{a}^\ve\nabla\tilde{\phi}^\ve\cdot\nabla\tilde{\phi}^\ve+\tilde{f}^\ve\cdot\nabla\tilde{\phi}^\ve\right)\chi
\\ \nonumber & \quad -\int_{A_{K,\ve}}\left(\tilde{a}^\ve\nabla\tilde{\phi}^\ve\cdot\nabla\tilde{\phi}^\ve+\tilde{f}^\ve\cdot\nabla\tilde{\phi}^\ve\right)\chi.
\end{align}
For the first term on the righthand side of \eqref{hc_92}, it follows from Lemma~\ref{lem_main} and the definition of $\chi$ that
\begin{equation}\label{hc_093}\lim_{\ve\rightarrow 0}\int_{\mcC_2}\left(\tilde{a}^\ve\nabla\tilde{\phi}^\ve\cdot\nabla\tilde{\phi}^\ve+\tilde{f}^\ve\cdot\nabla\tilde{\phi}^\ve\right)\chi = \E\left[A\Phi\cdot\Phi+F\cdot\Phi\right]\;\;\textrm{in probability.}\end{equation}
For the second term, we perform a second decomposition based on the size of $\nabla\tilde{\phi}^\ve$.  For every $M\in\N$ let $B_{M,\ve}$ be the set
\[B_{M,\ve} = \mcC_2\cap\left\{(x,t)\colon \abs{\nabla\tilde{\phi}^\ve(x,t)}\geq M\right\},\]
for which we have using the boundedness of $A$ and $F$ that, for some $c\in(0,\infty)$ independent of $\ve$, $K$, and $M$,
\begin{align*}
 &  \abs{\int_{A_{K,\ve}}\left(\tilde{a}^\ve\nabla\tilde{\phi}^\ve\cdot\nabla\tilde{\phi}^\ve+\tilde{f}^\ve\cdot\nabla\tilde{\phi}^\ve\right)\chi}
\\ & \leq \frac{cM^2}{K^2}\int_{\mcC_2} (\tilde{\phi}^\ve(x,t)-\langle\tilde{\phi}^\ve\rangle_{\mcC_2})^2 +\abs{\int_{\mcC_2}\left(\tilde{a}^\ve\nabla\tilde{\phi}^\ve\cdot\nabla\tilde{\phi}^\ve+\tilde{f}^\ve\cdot\nabla\tilde{\phi}^\ve\right)\mathbf{1}_{B_{M,\ve}}\chi}
\\ & \leq \frac{cM^2}{K^2}\int_{\mcC_2} (\tilde{\phi}^\ve(x,t)-\langle\tilde{\phi}^\ve\rangle_{\mcC_2})^2 +c\abs{\int_{\mcC_2}\abs{\nabla\tilde{\phi}^\ve}^2\mathbf{1}_{B_{M,\ve}}\chi}.
\end{align*}
It remains to treat the righthand side of \eqref{hc_91}.  It follows from the definition of $\chi$ and H\"older's inequality that, for some $c\in(0,\infty)$ depending on $\chi$,
\begin{align}\label{hc_94}
& \abs{\int_{\mcC_2}\tilde{s}^\ve\nabla\tilde{\phi}^\ve\psi_K(\tilde{\phi}^\ve(x,t)-\langle\tilde{\phi}^\ve\rangle_{\mcC_2})\cdot \nabla\chi}
\\ \nonumber & \leq \left(\int_{\mcC_2}\abs{\tilde{s}^\ve}^2\abs{\nabla\chi}\right)^\frac{1}{2}\left(\int_{\mcC_2}\abs{\nabla\tilde{\phi}^\ve}^2\psi_K(\tilde{\phi}^\ve(x,t)-\langle\tilde{\phi}^\ve\rangle_{\mcC_2})^2\abs{\nabla\chi}\right)^\frac{1}{2}
\\ \nonumber & \leq M\left(\int_{\mcC_2}\abs{\tilde{s}^\ve}^2\abs{\nabla\chi}\right)^\frac{1}{2}\left(\int_{\mcC_2}\psi_K(\tilde{\phi}^\ve(x,t)-\langle\tilde{\phi}^\ve\rangle_{\mcC_2})^2\abs{\nabla\chi}\right)^\frac{1}{2}
\\ \nonumber & \quad +K\left(\int_{\mcC_2}\abs{\tilde{s}^\ve}\abs{\nabla\chi}^2\right)^\frac{1}{2}\left(\int_{\mcC_2}\abs{\nabla\tilde{\phi}^\ve}^2\mathbf{1}_{B_{M,\ve}}\abs{\nabla\chi}\right)^\frac{1}{2}.
\end{align}
By the stationarity of the gradient, it follows that that the indicator functions of the sets $B_{M,\ve}$ are stationary and therefore by Lemma~\ref{lem_main} and Proposition~\ref{prop_sublinear} we have that, for every $K,M\in\N$, for some $c\in(0,\infty)$ depending on $\chi$,
\[\lim_{\ve\rightarrow 0}\left(\frac{cM^2}{K^2}\int_{\mcC_2} (\tilde{\phi}^\ve(x,t)-\langle\tilde{\phi}^\ve\rangle_{\mcC_2})^2 +c\abs{\int_{\mcC_2}\abs{\nabla\tilde{\phi}^\ve}^2\mathbf{1}_{B_{M,\ve}}}\right) = c\E\left[\abs{\Phi}^2\mathbf{1}_{\{\abs{\Phi}\geq M\}}\right]\;\;\textrm{in probability.}\]
Similarly, for the righthand side of \eqref{hc_94}, we have from Lemma~\ref{lem_main} and Proposition~\ref{prop_sublinear} that
\[\lim_{\ve\rightarrow 0}\left(M\left(\int_{\mcC_2}\abs{\tilde{s}^\ve}^2\abs{\nabla\chi}\right)^\frac{1}{2}\left(\int_{\mcC_2}\psi_K(\tilde{\phi}^\ve(x,t)-\langle\tilde{\phi}^\ve\rangle_{\mcC_2})^2\abs{\nabla\chi}\right)^\frac{1}{2}\right) = 0\;\;\textrm{in probability,}\]
and it follows from Lemma~\ref{lem_main} and Proposition~\ref{prop_sublinear} that, for some $c\in(0,\infty)$ depending on $\chi$,
\[\lim_{\ve\rightarrow 0}K\left(\int_{\mcC_2}\abs{\tilde{s}^\ve}\abs{\nabla\chi}^2\right)^\frac{1}{2}\left(\int_{\mcC_2}\abs{\nabla\tilde{\phi}^\ve}^2\mathbf{1}_{B_{M,\ve}}\abs{\nabla\chi}\right)^\frac{1}{2} = cK\E[\abs{S}^2]^\frac{1}{2}\E\left[\abs{\Phi}^2\mathbf{1}_{\{\abs{\Phi}\geq M\}}\right],\]
in probability.  After passing $M\rightarrow\infty$, it follows from the dominated convergence theorem that
\begin{equation}\label{hc_95} \lim_{\ve\rightarrow 0}\abs{\int_{A_{K,\ve}}\left(\tilde{a}^\ve\nabla\tilde{\phi}^\ve\cdot\nabla\tilde{\phi}^\ve+\tilde{f}^\ve\cdot\nabla\tilde{\phi}^\ve\right)\chi} = 0\;\;\textrm{in probability,}\end{equation}
and
\begin{equation}\label{hc_96}\lim_{\ve\rightarrow 0}\abs{\int_{\mcC_2}\tilde{s}^\ve\nabla\tilde{\phi}^\ve\psi_K(\tilde{\phi}^\ve(x,t)-\langle\tilde{\phi}^\ve\rangle_{\mcC_2})\cdot \nabla\chi} = 0\;\;\textrm{in probability.}\end{equation}
Finally, for the lefthand side of \eqref{hc_90}, it follows from Proposition~\ref{prop_sublinear} and the definition of $\Psi_K$ that, for some $c\in(0,\infty)$ depending on $\chi$,
\begin{equation}\label{hc_97} \limsup_{\ve\rightarrow 0}\abs{\int_{\mcC_2}\Psi_K(\tilde{\phi}^\ve(x,t)-\langle\tilde{\phi}^\ve\rangle_{\mcC_2})\partial_t\chi}\leq c\limsup_{\ve\rightarrow 0}\int_{\mcC_2}(\tilde{\phi}^\ve(x,t)-\langle\tilde{\phi}^\ve\rangle_{\mcC_2})^2 = 0\;\;\textrm{in probability.}\end{equation}
In combination \eqref{hc_093}, \eqref{hc_95}, \eqref{hc_96}, and \eqref{hc_97} prove the energy equality
\[\E\left[A\Phi\cdot\Phi+F\cdot\Phi\right] =0,\]
which completes the proof of \eqref{energy}.  The proof of uniqueness is then a consequence of the energy equality and the uniformly ellipticity, since any two solutions $\Phi_1$ and $\Phi_2$ will satisfy
\[\lambda\E[\abs{\Phi_1-\Phi_2}^2]\leq \E[A\Phi\cdot\Phi]=0,\]
which completes the proof.  \end{proof}

\subsection{The correctors and transpose correctors}\label{sec_transpose}

In this section, we define for once the homogenization correctors and briefly collect some information about the transpose correctors that will be used in the proof of homogenization below.  These are the correctors corresponding to the transpose of the diffusion matrix, and formally solve the equation, for every $i\in\{1,\ldots,d\}$,
\[-\partial_t\phi^t_i = \nabla\cdot(a^t-s)\nabla\phi^t_i+\underline{b}\cdot \nabla\phi^i_t,\]
where time has been reversed and the diffusive part of the equation has been transposed, but the noise entering the equation remains the same.  Observe, in particular, that the sign of the time derivative and the relationship between $S$ and $\underline{B}$ played no role at all in the proofs of Propositions~\ref{prop_cor_1}, \ref{prop_sublinear}, and \ref{prop_unique} above.

\begin{prop}\label{prop_tcorrector} Assume~\eqref{steady}.  Then for every $F\in L^2(\O;R^d)$ there exists a unique $\Phi^t\in L^2_{\textrm{pot}}(\O)$ that satisfies the distributional equalities
\[-D_0\Phi^t_i = D_i\left( D\cdot (A^t-S)\Phi^t+\underline{B}\cdot \Phi^t + D\cdot F\right)\;\;\textrm{on $\O$ for every}\;\;i\in\{1,\ldots,d\},\]
which $\P$-a.e.\ implies the existence of a unique $\phi^t\in L^2_{\textrm{loc}}(\R;H^1_{\textrm{loc}}(\R^d))$ that satisfies
\[\nabla\phi^t(x,t) = \Phi^t(\tau_{x,t}\o)\;\;\textrm{and}\;\;\fint_{\mathcal{C}_1}\phi^t = 0,\]
and that satisfies distributionally
\[-\partial_t\phi^t= \nabla\cdot(a^t-s)\nabla\phi^t+\underline{b}\cdot\nabla\phi^t+\nabla\cdot f,\]
for $f(x,t)=F(\tau_{x,t}\o)$.  Furthermore, for every $\ve\in(0,1)$, for $w^\ve(t) = \ve^{-1}\int_0^t\underline{b}(\nicefrac{s}{\ve^2})\ds$, for $\phi^{t,\ve}(x,t)=\phi^t(\nicefrac{x}{\ve},\nicefrac{t}{\ve^2})$, and for $\tilde{\phi}^{t,\ve}(x,t)=\phi^{t,\ve}(x+w^\ve(t),t)$, we have that, for every $R\in(0,\infty)$,
\[\lim_{\ve\rightarrow 0}\int_{\mcC_R}\ve^2(\tilde{\phi}^{t,\ve}-\langle\tilde{\phi}^{t,\ve}\rangle_{\mcC_R})^2 = 0\;\;\textrm{in probability,}\]
for $\mathcal{C}_R = B_R\times [0,R^2]$ and for $\langle\tilde{\phi}^{t,\ve}\rangle_{\mcC_R}=\fint_{\mcC_R}\tilde{\phi}^{t,\ve}$.  \end{prop}

\begin{proof}  The proof is identical to the proofs of Propositions~\ref{prop_cor_1}, \ref{prop_sublinear}, and \ref{prop_unique}.  \end{proof}

\begin{remark}  Observe that the functions $\phi^\ve$ are defined differently in Proposition~\ref{prop_tcorrector} than they were in Proposition~\ref{prop_sublinear}.  In the notation of Proposition~\ref{prop_tcorrector}, in order to simplify the notation in the argument of Proposition~\ref{prop_sublinear} we made the choice to work directly with $\ve\phi^\ve$ as opposed to $\phi^\ve$ itself.  For the remainder of the paper we will take the convention of Proposition~\ref{prop_tcorrector} and Definition~\ref{def_correctors} below.  \end{remark}

\begin{definition}\label{def_correctors}  For every $i\in\{1,\ldots,d\}$ let $\Phi_i, \Phi_i^t\in L^2_{\textrm{pot}}(\O)$ denote the unique random gradients corresponding to $F = (A+S)e_i$ and $F=(A^t-S)e_i$ constructed in Proposition~\ref{prop_cor_1} and Proposition~\ref{prop_tcorrector} respectively.  Let $\phi_i$ and $\phi_i^t$ be $\P$-a.e.\ defined as the unique $L^2_{\textrm{loc}}(\R;H^1_{\textrm{loc}}(\R^d))$ functions satisfying
\[\nabla\phi(x,t) = \Phi_i(\tau_{x,t})\;\;\textrm{and}\;\;\nabla\phi_i^t(x,t) = \Phi^t_i(\tau_{x,t}\o)\;\;\textrm{and}\;\;\fint_{\mcC_1}\phi_i = \fint_{\mcC_1}\phi_i^t = 0,\]
and satisfying distributionally
\[\partial_t\phi_i = \nabla\cdot(a+s)(\nabla\phi_i+e_i)+\underline{b}\cdot\nabla\phi_i\;\;\textrm{and}\;\; -\partial_t\phi^t_i = \nabla\cdot(a^t-s)(\nabla\phi^t_i+e_i)+\underline{b}\cdot\nabla\phi^t_i.\]
Finally, for every $\xi\in\R^d$ define $\Phi_\xi = \sum_{i=1}^d\xi_i\Phi_i$ and $\Phi^t_\xi = \sum_{i=1}^d\xi_i\Phi^t_i$ and let $\phi_\xi$ and $\phi^t_\xi$ be defined as above for $e_i=\xi$. \end{definition}

\subsection{The homogenized coefficient}\label{sec_hom}  In this section, we use the homogenization correctors to characterize the effective coefficient $\overline{a}$.  The characterization of $\overline{a}^t$ is important for the proof of homogenization, where in the perturbed test function method it is necessary to use the transpose homogenization correctors.

\begin{prop}\label{prop_hom_mat}  Assume~\eqref{steady} and let the correctors be as in Definition~\ref{def_correctors}.  Let $\overline{a},\overline{m}\in\R^{d\times d}$ be defined by
\[\overline{a}\xi = \E\left[(A+S)(\Phi_\xi+\xi)\right]\;\;\textrm{and}\;\;\overline{m} = \E\left[(A^t-S)(\Phi^t_\xi+\xi)\right].\]
Then $\overline{a}$ and $\overline{m}$ are uniformly elliptic in the sense that both satisfy, for every $\xi\in\R^d$,
\[\lambda\abs{\xi}^2\leq \langle a\xi,\xi\rangle\;\;\textrm{and}\;\;\abs{\overline{a}\xi}\leq \abs{\xi}\left(\Lambda+\E[\abs{S}^2]^\frac{1}{2}\right)\left(1+\left(\sum_{i=1}^d\E[\abs{\Phi_i}^2]\right)^\frac{1}{2}\right).\]
Furthermore, we have that
\[\overline{a}=\overline{m}^t.\]
\end{prop}

\begin{proof}   It follows from Proposition~\ref{prop_unique}, the energy inequality \eqref{energy}, the uniform ellipticity \eqref{intro_elliptic}, an approximation argument first approximating $S$ by bounded, skew-symmetric matrices as in Proposition~\ref{prop_cor_1} and using the fact that the homogenized matrix is stable with respect to strong convergence of the coefficients and weak convergence of the gradients, Jensen's inequality, and $\E[\Phi_\xi]=0$ that
\[\langle \overline{a}\xi,\xi\rangle = \E\left[(A+S)(\Phi_\xi+\xi)\cdot\xi\right] = \E\left[(A+S)(\Phi_\xi+\xi)\cdot(\Phi_\xi+\xi)\right]\geq \lambda \E\left[\abs{\Phi_\xi+\xi}^2\right]\geq \lambda\abs{\xi}^2.\]
For the second inequality, it follows from H\"older's inequality that
\begin{align*}
\abs{\overline{a}\xi}=\abs{\E[(A+S)(\Phi_\xi+\xi)]} & \leq \E[(\Lambda+\abs{S})(\abs{\Phi_\xi}+\abs{\xi})]
\\ & \leq \abs{\xi}\left(\Lambda+\E[\abs{S}^2]^\frac{1}{2}\right)\left(1+\left(\sum_{i=1}^d\E[\abs{\Phi_i}^2]\right)^\frac{1}{2}\right).
\end{align*}
For the final equality, again working first on an approximate level using the fact that the homogenized matrices are stable with respect to strong convergence of the coefficients and weak convergence of the gradients, we observe $\P$-a.e.\ that, for every $i,j\in\{1,\ldots,d\}$,
\[\partial_t(\tilde{\phi}^\ve_i-\tilde{\phi}^{t,\ve}_j) = \nabla\cdot(\tilde{a}^\ve+\tilde{s}^\ve)(\nabla\tilde{\phi}^\ve_i+e_i)+\nabla\cdot(\tilde{a}^{t,\ve}-\tilde{s}^\ve)(\nabla\tilde{\phi}^{t,\ve}_j+e_j).\]
A repetition of the proof of Proposition~\ref{prop_sublinear} then proves that, on the approximate level, for every $i,j\in\{1,\ldots,d\}$,
\[\E[(A+S)(\Phi_i+e_i)(\Phi_i-\Phi^t_j) ]+ \E[(A^t-S)(\Phi^j_t+e_j)(\Phi_i-\Phi^t_j)]=0.\]
Since it follows from the energy equality \eqref{energy} that
\[\E[A(\Phi_i+e_i)\Phi_i] = \E[A^t(\Phi^t_j+e_j)\Phi^t_j],\]
we conclude that
\[\E[(A+S)(\Phi_i+e_i)\Phi^t_j] = \E[(A^t-S)(\Phi^j_t+e_j)\Phi_i].\]
Therefore, for every $i,j\in\{1,\ldots,d\}$,
\begin{align*}
\overline{a}_{ij} & = \E[(A+S)(\Phi_i+e_i)e_j]
\\ & = \E[(A+S)(\Phi_i+e_i)(\Phi^t_j+e_j)]-\E[(A+S)(\Phi_i+e_i)\Phi^t_j]
\\ & = \E\left[(A^t-S)(\Phi^t_j+e_j)(\Phi_i+e_i)\right]-\E[(A^t-S)(\Phi^j_t+e_j)\Phi_i]
\\ & = \E\left[(A^t-S)(\Phi^t_j+e_j)e_i\right] = \overline{m}_{ji},
\end{align*}
which completes the proof.
\end{proof}

\section{The proof of homogenization}\label{sec_homogenize}  In this section, we prove the homogenization of \eqref{intro_neq} in law.  We first recall some well-known facts about the solution theories of \eqref{intro_neq} and \eqref{intro_spde} in Section~\ref{sec_wp} and we prove the homogenization in Section~\ref{sec_main_hom} using a variation of the perturbed test function method.

\subsection{Well-posedness of the \eqref{intro_neq} and \eqref{intro_spde}}\label{sec_wp}  In this section, we collect the standard well-posedness results for \eqref{intro_neq} and \eqref{intro_spde}.

\begin{prop}\label{prop_well_posed}  Assume~\eqref{steady} and let $T\in(0,\infty)$, let $\ve\in(0,1)$, let $f\in(L^2\cap \C)(\R^d\times(0,T])$, and let $g\in(L^2\cap\C)(\R^d)$.  Then there exists a unique solution $\rho^\ve\in L^2([0,T];H^1(\R^d))\cap\C([0,T];L^2(\R^d))$ of the equation
\[\partial_t\rho^\ve = \nabla \cdot (a^\ve+s^\ve) \nabla\rho^\ve+\ve^{-1}\underline{b}^\ve\cdot \nabla\rho^\ve \;\;\textrm{in}\;\;\R^d\times (0,T)\;\;\textrm{with}\;\;\rho^\ve(\cdot,0)=\rho_0.\]
Furthermore, the solution $\rho^\ve$ satisfies the estimates
\[\norm{\rho^\ve}_{L^\infty(\R^d\times[0,T])}\leq \norm{g}_{L^\infty(\R^d)}+T\norm{f}_{L^\infty(\R^d\times[0,T])},\]
and, for $c\in(0,\infty)$ depending on the ellipticity constants and $T$,
\[\max_{t\in[0,T]}\norm{\rho^\ve}^2_{L^2(\R^d)}+\int_0^T\int_{\R^d}\abs{\nabla\rho^\ve}^2 \leq c\left(\norm{g}^2_{L^2(\R^d)}+\norm{f}^2_{L^2(\R^d\times[0,T])}\right).\]
\end{prop}
\begin{proof}  The proof of existence follows from a Galerkin approximation, the $L^\infty$-estimate is a consequence of the comparison principle using the fact that the functions $s_{\pm}(t) = \norm{g}_{L^\infty(\R^d)}\pm \norm{f}_{L^\infty(\R^d\times[0,T])}$ are respectively a subsolution and a supersolution of the equation, and the final estimate is the usual energy inequality obtained by testing the equation against $\rho^\ve \chi_R$ for a smooth cutoff $\chi_R$ and then passing to the limit $R\rightarrow \infty$ using the $L^2$-integrability and $L^\infty$-boundedness of the solution and applying Gr\"onwall's inequality.  For full details see Evans \cite[Chapter~7]{Eva2010}.  \end{proof}

\begin{definition}  Assume~\eqref{steady}, let $(\F_t)_{t\in[0,\infty)}$ be a filtration on $(\O,\F)$, let $B$ be a standard $\F_t$-adapted, $d$-dimensional Brownian motion on $(\O,\F)$, let $T\in(0,\infty)$, let $f\in(L^2\cap \C)(\R^d\times(0,T])$, and let $g\in(L^2\cap\C)(\R^d)$.  A strong solution of \eqref{intro_spde} is a continuous, $\F_t$-adapted, $L^2$-valued process $\overline{\rho}\in L^2(\O\times[0,T];H^1(\R^d))$ that satisfies, for every $\psi\in\C^\infty_c(\R^d\times[0,T))$,
\[-\int_0^T\int_{\R^d}\rho^\ve\partial_t\psi = \int_{\R^d}g\psi(\cdot,0)-\int_0^T\int_{\R^d}\overline{a}\nabla\overline{\rho}\cdot\nabla\psi+\int_0^T\int_{\R^d}\psi\nabla\overline{\rho}\circ\Sigma\dd B_t+\int_0^T\int_{B_n}f\psi,\]
where $\circ$ denotes a Stratonovich integral.  \end{definition}

\begin{prop}\label{prop_wp_spde}  Assume~\eqref{steady}, let $(\F_t)_{t\in[0,\infty)}$ be a filtration on $(\O,\F)$, let $B$ be standard $\F_t$-adapted, $d$-dimensional Brownian motion on $(\O,\F)$, let $T\in(0,\infty)$, let $\ve\in(0,1)$, let $f\in(L^2\cap \C)(\R^d\times(0,T])$, and let $g\in(L^2\cap\C)(\R^d)$.  Then there exists a unique strong solution of the equation
\[\partial_t\overline{\rho} = \nabla\cdot\overline{a}\nabla\overline{\rho}+\nabla\overline{\rho}\circ\Sigma \dd B_t+f \;\;\textrm{in}\;\;\R^d\times(0,\infty)\;\;\textrm{with}\;\;\overline{\rho}(\cdot,0)=g.\]
Furthermore, the solution $\overline{\rho}$ satisfies $\P$-a.e.\ that
\[\norm{\overline{\rho}}_{L^\infty(\R^d\times[0,T])}\leq \norm{g}_{L^\infty(\R^d)}+T\norm{f}_{L^\infty(\R^d\times[0,T])},\]
and, for $c\in(0,\infty)$ depending on the ellipticity constants and $T$,
\[\max_{t\in[0,T]}\norm{\overline{\rho}}^2_{L^2(\R^d)}+\int_0^T\int_{\R^d}\abs{\nabla\rho^\ve}^2 \leq c\left(\norm{g}^2_{L^2(\R^d)}+\norm{f}^2_{L^2(\R^d\times[0,T])}\right).\]
\end{prop}

\begin{proof}  The proof can be established directly on the level of the SPDE using a Galerkin approximation, the subsolution and supersolution  $s_{\pm}(t) = \norm{g}_{L^\infty(\R^d)}\pm \norm{f}_{L^\infty(\R^d\times[0,T])}$ and the comparison principle, and It\^o's formula (see, for example, Krylov \cite[Theorem~3.1]{Kry2013}) applied to the process $\overline{\rho}^2$.  Alternately, one can also immediately apply It\^o's formula \cite{Kry2013} to the function $\tilde{\overline{\rho}}(x,t) = \overline{\rho}(x+\Sigma B_t,t)$ which due to the Stratonovich form of the noise reduces the SPDE to a PDE with random coefficients.  The corresponding PDE can be handled as in Proposition~\ref{prop_well_posed} to complete the proof.  \end{proof}

\subsection{The proof of homogenization}\label{sec_main_hom}  We will now prove the homogenization of \eqref{intro_neq}.  The proof is based on a refinement of the perturbed test function method, where in this case, due to the low integrability of the drift and stream matrix, it is necessary to introduce stationary approximations of the corrector that introduce several errors that must be controlled.  We the Skorokhod representation theorem to establish the convergence of the $\rho^\ve$ in law.

\begin{thm}\label{thm_final}  Assume~\eqref{steady} and let $T\in(0,\infty)$, let $\ve\in(0,1)$, let $f\in(L^2\cap \C)(\R^d\times(0,T])$, and let $g\in(L^2\cap\C)(\R^d)$.  Then, as $\ve\rightarrow 0$, the solutions $\rho^\ve$ of the equation
\begin{equation}\label{hom_0}\partial_t\rho^\ve = \nabla \cdot (a^\ve+s^\ve) \nabla\rho^\ve+\ve^{-1}\underline{b}^\ve\cdot \nabla\rho^\ve+f \;\;\textrm{in}\;\;\R^d\times (0,T)\;\;\textrm{with}\;\;\rho^\ve(\cdot,0)=g,\end{equation}
converge in law on $L^2([0,T];H^1(\R^d))\cap\C([0,T];L^2(\R^d))$ to a solution of the equation
\[\partial_t\overline{\rho} = \nabla\cdot\overline{a}\nabla\overline{\rho}+\nabla\overline{\rho}\circ\Sigma \dd B_t+f \;\;\textrm{in}\;\;\R^d\times(0,\infty)\;\;\textrm{with}\;\;\overline{\rho}(\cdot,0)=g,\]
for a standard $d$-dimensional Brownian motion $B$.  \end{thm}

\begin{proof}  For technical reasons that will be described below, it is necessary to introduce a stationary approximation of the transpose homogenization correctors.  For every $\d\in(0,1)$ and $i\in\{1,\ldots,d\}$ let $\Phi^t_{i,\d}\in \mathcal{H}^1(\O)\cap L^2(\O)$ be the unique stationary function satisfying that the vector
\[(-(D\cdot (A^t-S)(D\Phi^t_{i,\d}+e_i)+\underline{B}\cdot D\Phi^t_{i,\d}+\delta \Phi^t_{i,\d}), D\Phi^t_{i,\d}),\]
is space-time distributionally curl-free.  The functions are constructed using Proposition~\ref{prop_cor_1} using the fact that the additional coercivity means that the solutions exist as stationary functions.  It then follows that the $\phi^t_{i,\d}(t,x) = \Phi^t_{i,\d}(\tau_{x,t}\o)$ is a distributional solution of the equation
\[\d\phi^t_{i,\d} -\partial_t\phi^t_{i,\d} = \nabla\cdot(a^t-s)(\nabla\phi^t_{i,\d}+e_i)+\underline{b}\cdot\nabla\phi^t_{i,\d}.\]
For every $\ve\in(0,1)$ let $w^\ve = \ve^{-1}\int_0^t\underline{b}(\nicefrac{s}{\ve^2})\ds$, let $\phi^{t,\ve}_{i,\d}(x,t) = \phi^t_{i,\d}(\nicefrac{x}{\ve},\nicefrac{t}{\ve^2})$, and let $\tilde{\phi}^{t,\ve}_{i,\d}(x,t) = \phi^{t,\ve}_{i,\d}(x+w^\ve(t),t)$.  We have that $\tilde{\phi}^{t,\ve}_{i,\d}$ is a distributional solution of the equation
\[\ve^{-2}\d\tilde{\phi}^{t,\ve}_{i,\d}-\partial_t \tilde{\phi}^{t,\ve}_{i,\d}  = \nabla\cdot(\tilde{a}^{t,\ve}-\tilde{s}^\ve)(\nabla \tilde{\phi}^{t,\ve}_{i,\d} +e_i).\]
For every $R\in(0,\infty)$ it follows from a repetition of the proof of the energy inequality \eqref{energy} in Proposition~\ref{prop_unique} that there exists $c\in(0,\infty)$ depending on $R$ such that, $\P$-a.e. along a subsequence $\ve\rightarrow 0$,
\[\limsup_{\ve\rightarrow 0} \int_{\mcC_R}\ve^{-2}\d(\tilde{\phi}^{t,\ve}_{i,\d} )^2 \leq -\E\left[A^t D\Phi^t_{i,\d}\cdot D\Phi^t_{i,\d}+(A^t-S)e_i\cdot D\Phi^t_{i,\d}\right]<\infty. \]
We therefore conclude that, $\P$-a.e. along a subsequence $\ve\rightarrow 0$, the functions
\begin{equation}\label{hom_1}\ve^{-1}\d^\frac{1}{2}\tilde{\phi}^{t,\ve}_{i,\d}\;\;\textrm{are bounded uniformly in $\ve\in(0,1)$ in $L^2(\mcC_R)$.}\end{equation}
We will now proceed with the proof of homogenization via the perturbed test function method.

Let $\rho^\ve$ be the solution of \eqref{hom_0}, let $\tilde{\rho}^\ve(x,t) = \rho^\ve(x+w^\ve(t),t)$, and let $\psi\in\C^\infty_c(\R^d\times[0,T)$ be fixed.  For every $\ve,\d\in(0,1)$ define the perturbed test function $\psi^{\ve,\d}$ as
\[\psi^{\ve,\d} = \psi+\ve \tilde{\phi}^{t,\ve}_{i,\d}\partial_i\psi,\]
using Einstein's summation convention over repeated indices, and for the approximate transpose correctors defined above.  It then follows from an approximation argument, using that the final equation is stable with respect to strong convergence of the functions and coefficients, and weak convergence of the gradients, that
\begin{align*}
& -\int_0^T\int_{\R^d}\tilde{\rho}^\ve\partial_t\psi^{\ve,\d}-\int_{\R^d}g(x)\psi^{\ve,\d}(x,0)= -\int_0^T\int_{\R^d}(\tilde{a}^\ve+\tilde{s}^\ve)\nabla\tilde{\rho}^\ve\cdot\nabla\psi^{\ve,\d}+\int_0^T\int_{\R^d}\tilde{f}^\ve\psi^{\ve,\d}.
\end{align*}
Then, using the definition of $\psi^{\ve,\d}$ and the equation satisfied by $\tilde{\phi}^{t,\ve}_{i,\d}$,
\begin{align*}
& -\int_0^T\int_{\R^d}\tilde{\rho}^\ve\partial_t\psi-\int_0^T\int_{\R^d}\tilde{\rho}^\ve\ve\tilde{\phi}^{t,\ve}_{i,\d}\partial_t\partial_i\psi- \int_0^T\int_{\R^d}\ve^{-1}\d \tilde{\phi}^{t,\ve}_{i,\d}\tilde{\rho}^\ve\partial_i\psi-\int_{\R^d}g(x)\psi^{\ve,\d}(x,0)
\\ &  = \int_0^T\int_{\R^d}\tilde{\rho}^\ve (\tilde{a}^{t,\ve}-\tilde{s}^{t,\ve})(\nabla\tilde{\phi}^{t,\ve}_{i,\d}+e_i)\cdot \nabla\partial_i\psi - \int_0^T\int_{\R^d}\tilde{a}^\ve\nabla\tilde{\rho}^\ve\cdot\nabla\partial_i\psi \ve\tilde{\phi}^{t,\ve}_{i,\d}
\\ & \quad -\int_0^T\int_{\R^d}\tilde{s}^\ve\nabla\tilde{\rho}^\ve\cdot\nabla\partial_i\psi \ve \tilde{\phi}^{t,\ve}_{i,\d}+ \int_0^T\int_{\R^d}\tilde{f}^\ve\psi^{\ve,\d}.
\end{align*}
We then write, using the skew-symmetry of $s$,
\[\int_0^T\int_{\R^d}\tilde{s}^\ve\nabla\tilde{\rho}^\ve\cdot\nabla\partial_i\psi \tilde{\phi}^{t,\ve}_{i,\d} = -\int_0^T\int_{\R^d}\tilde{\rho}^\ve \nabla\cdot (\tilde{s}^\ve \ve \tilde{\phi}^{t,\ve}_{i,\d})\cdot\nabla\partial_i\psi,\]
to conclude that
\begin{align*}
& -\int_0^T\int_{\R^d}\tilde{\rho}^\ve\partial_t\psi-\int_0^T\int_{\R^d}\tilde{\rho}^\ve\ve\tilde{\phi}^{t,\ve}_{i,\d}\partial_t\partial_i\psi- \int_0^T\int_{\R^d}\ve^{-1}\d \tilde{\phi}^{t,\ve}_{i,\d}\tilde{\rho}^\ve\partial_i\psi-\int_{\R^d}g(x)\psi^{\ve,\d}(x,0)
\\ \nonumber &  = \int_0^T\int_{\R^d}\tilde{\rho}^\ve (\tilde{a}^{t,\ve}-\tilde{s}^{t,\ve})(\nabla \tilde{\phi}^{t,\ve}_{i,\d}+e_i)\cdot \nabla\partial_i\psi - \int_0^T\int_{\R^d}\tilde{a}^\ve\nabla\tilde{\rho}^\ve\cdot\nabla\partial_i\psi \ve\tilde{\phi}^{t,\ve}_{i,\d}
\\ \nonumber & \quad +\int_0^T\int_{\R^d}\tilde{\rho}^\ve \nabla\cdot (\tilde{s}^\ve \ve \tilde{\phi}^{t,\ve}_{i,\d})\cdot\nabla\partial_i\psi+ \int_0^T\int_{\R^d}\tilde{f}^\ve\psi^{\ve,\d}.
\end{align*}
We emphasize again that every term of the above equation is well-defined using the $L^\infty$-boundedness of $\tilde{\rho}^\ve$ and $\psi\in\C^\infty_c(\R^d\times[0,\infty))$.  It is therefore possible to first justify all of the above computations on an approximate level, and then use the strong convergence of the coefficients and approximate correctors and solutions, and the weak convergence of the gradients to pass to the limit.

The proof of homogenization will now follow from an argument analogous to that used in Lemma~\ref{lem_main}, once we have shown the tightness of all of the relevant random variables.  Let $n\in\N$.  It follows from Proposition~\ref{prop_well_posed} that the $\tilde{\rho}^\ve$ are uniformly bounded in $L^2(B_n\times[0,T])\cap L^2([0,T];H^1(B_n))$ and using the Sobolev embedding theorem that $\partial_t\tilde{\rho}^\ve$ is uniformly bounded in $L^1([0,T];H^{-s}(B_n))$ for every $s>\nicefrac{d}{2}+1$.  It then follows from the Aubin-Lions-Simon lemma \cite{Aub1963,Lio1969,Sim1987}, the compact embedding of $H^1(B_n)$ into $L^2(B_n)$, and the continuous embedding of $L^2(B_n)$ into $H^{-s}(B_n)$ that the laws of the $\rho^\ve$ are tight in the strong topology of $L^2([0,T];L^2(B_n))$, and the laws of
\[\left(\frac{\rho^\ve}{\norm{\rho^\ve}_{L^2([0,T];H^1(B_n))}}\mathbf{1}_{\left\{\rho^\ve \neq 0\right\}}, \norm{\rho^\ve}_{L^2([0,T];H^1(B_n))}\right)\;\;\textrm{are tight in}\;\;B_1(L^2([0,T];H^1(B_n)))\times \R,\]
in the metric topology of weak convergence on the unit ball $B_1(L^2([0,T];H^1(B_n)))$.  Let $m\in\N$ and let $\delta_m = m^{-1}$  It follows from \eqref{hom_1} that the laws of
\[\left(\frac{\ve^{-1}{\delta_m}^\frac{1}{2}\tilde{\phi}^{t,\ve}_{i,\d_m}}{\norm{\ve^{-1}{\delta_m}^\frac{1}{2}\tilde{\phi}^{t,\ve}_{i,\d_m}}_{L^2(B_n\times[0,T])}}\mathbf{1}_{\left\{\tilde{\phi}^{t,\ve}_{i,\d} \neq 0\right\}}, \norm{\ve^{-1}{\delta_m}^\frac{1}{2}\tilde{\phi}^{t,\ve}_{i,\d_m}}_{L^2(B_n\times[0,T])}\right),\]
are tight in $B_1(L^2(B_n\times[0,T]))\times \R$ in the metric weak topology of $B_1(L^2(B_n\times[0,T]))$.  We have that the laws of the $\ve\tilde{\phi}^{t,\ve}_{i,\d_m}$ are tight in the strong topology of $L^2(B_n\times[0,T])$ using either the stationarity or Proposition~\ref{prop_sublinear}.  We repeat the identical argument of Lemma~\ref{lem_main} to deduce that the laws of the random variables
\[\left(\frac{(\tilde{a}^{t,\ve}-\tilde{s}^{t,\ve})(\tilde{\phi}^{t,\ve}_{i,\d_m}+e_i)\dx\dt}{\norm{((\tilde{a}^{t,\ve}-\tilde{s}^{t,\ve})(\tilde{\phi}^{t,\ve}_{i,\d_m}+e_i)}_{L^1(B_n\times[0,T])}}, \norm{((\tilde{a}^{t,\ve}-\tilde{s}^{t,\ve})(\tilde{\phi}^{t,\ve}_{i,\d_m}+e_i)}_{L^1(B_n\times[0,T])}\right),\]
and the analogous random variables $\nabla\cdot(\tilde{s}^\ve\ve\tilde{\phi}^{t,\ve}_{i,\d_m})$ are tight in $\mathcal{M}(B_n\times[0,T])^d\times\R$.  We then consider the family of all such random variables, for every $i\in\{1,\ldots,d\}$,
\[X^{\ve,n,m} = (\rho^\ve,(\ve^{-1}\delta_m^\frac{1}{2}\tilde{\phi}^{t,\ve}_{i,\delta_m}),\ve\tilde{\phi}^{t,\ve}_{i,\d_m},(\tilde{a}^{t,\ve}-\tilde{s}^{t,\ve})(\tilde{\phi}^{t,\ve}_{i,\d_m}+e_i),\nabla\cdot(\tilde{s}^\ve\ve\tilde{\phi}^{t,\ve}_{i,\d_m})),\]
taking values in the produce metric space $\overline{X}^n$ defined by $B_n$ specified above that is independent of $m\in\N$.  It follows from the above that the laws are $(X^{\ve,n,m})_{\ve\in(0,1)}$ are tight on $X^n$ for every $n,m\in\N$.  Then, for every $\ve\in(0,1)$, we define $X^\ve = (X^{\ve,n,m})_{n,m\in\N}$ and conclude using the metric defined in Lemma~\ref{lem_main} that the laws of the $(X^\ve)_{\ve\in(0,1)}$ are tight in the space $\overline{X} = \prod_{n,m=1}^\infty \overline{X}^n$.

Let $\ve_k\rightarrow 0$ as $k\rightarrow\infty$ be a subsequence.  Then using Prokhorov's theorem and \eqref{intro_weak_con}, since the random variables $(X^{\ve_k},w^{\ve_k})$ are tight on the space $\overline{X}\times\C([0,T];\R^d)$, there exists a further subsequence still denoted $\ve_k\rightarrow 0$ as $k\rightarrow \infty$ and a probability measure $\mu$ on $\overline{X}\times\C([0,T];\R^d)$ such that, as $k\rightarrow\infty$,
\[(X^{\ve_k},w^{\ve_k})\rightharpoonup \mu\;\;\textrm{in distribution on $\overline{X}\times\C([0,T];\R^d)$.}\]
The remainder of the proof is then a consequence of the following considerations.  Suppose that for every $n,m\in\N$ we have that, along a subsequence $\ve\rightarrow 0$, the solutions $\tilde{\rho}^\ve$ were $\P$-a.e.\ converging strongly to some $\tilde{\overline{\rho}}$ in $L^2([0,T];L^2(B_n)$, that the $\tilde{\rho}^\ve$ were $\P$-a.e.\ converging weakly to the same $\tilde{\overline{\rho}}$ in $L^2([0,T];H^1(B_n))$, that the $w^\ve$ were $\P$-a.e.\ converging strongly to a Brownian motion $w$ in $\C([0,T];\R^d)$, that the $\ve\tilde{\phi}^{t,\ve}_{i,\d_m}$ were $\P$-a.e.\ converging strongly to zero in $L^2(B_n\times[0,T])$, and that the $\ve^{-1}\delta^\frac{1}{2}\tilde{\phi}^{t,\ve}_{i,\delta_m}$ are $\P$-a.e.\ uniformly bounded in $\ve\in(0,1)$ and $m\in\N$ in $L^2(B_n\times[0,T])$.  It would then follow that $\P$-a.e.\ the $\psi^{\ve,\d}$ converge strongly to $\psi$ in $L^2([0,T];L^2(B_n))$, and it would follow that, for some $c\in(0,\infty)$ depending on $\psi$ but  independent of $\d_m\in(0,1)$, for every $i\in\{1,\ldots,d\}$, 
\[\limsup_{\ve\rightarrow 0}\abs{\int_0^T\int_{\R^d}\ve^{-1}\d_m \tilde{\phi}^{t,\ve}_{i,\d_m}\tilde{\rho}^\ve\partial_i\psi}\leq c\delta_m^\frac{1}{2}.\]
It would follow as in Lemma~\ref{lem_main} that, $\P$-a.e.\ as $\ve\rightarrow 0$,
\[(\tilde{a}^{t,\ve}-\tilde{s}^{t,\ve})(\tilde{\phi}^{t,\ve}_{i,\d}+e_i)\rightharpoonup \overline{a}^{\d,t}\;\;\textrm{and}\;\;(\tilde{s}^\ve \ve \tilde{\phi}^{t,\ve}_{i,\d})\rightharpoonup \E\left[D(S D\Phi^t_{i,\d})\right]=0\;\;\textrm{weakly,}\]
where it is at this point that we use the stationarity of the $\tilde{\phi}^{t,\ve}_{i,\d}$ and where $\overline{a}^\d$ is the homogenized matrix defined by the $\Phi^t_{i,\d}$.  The separability of the space of compactly supported smooth functions in the sup-norm then proves that the limiting function $\tilde{\overline{\rho}}$ will satisfy, using Proposition~\ref{prop_hom_mat}, for every $\psi\in\C^\infty_c(\R^d\times[0,\infty))$, for $c\in(0,\infty)$ depending on $\psi$ but independent of $m$,
\[\abs{-\int_0^T\int_{\R^d}\tilde{\overline{\rho}}\partial_t\psi - \int_{\R^d}g\psi +\int_0^T\int_{\R^d} \overline{a}^\d\nabla\tilde{\overline{\rho}}\cdot\nabla\psi - \int_0^T\int_{\R^d}f(x+w(t),t)\psi }\leq c\d_m\;\;\textrm{with}\;\;\tilde{\overline{\rho}}(\cdot,0)=g.\]
Since it follows from the uniqueness of Proposition~\ref{prop_unique} that, as $m\rightarrow\infty$, we have that the $\Phi^t_{i,\d_m}$ converge weakly in $L^2_{\textrm{pot}}(\O)$ to $\Phi^t_i$, it follows that, as $m\rightarrow\infty$, we have that $\overline{a}^\d\rightarrow\overline{a}$.  Therefore, after passing $m\rightarrow\infty$ in the above equation, we have $\P$-a.e.\ that
\begin{equation}\label{final}\partial_t\tilde{\overline{\rho}}=\nabla\cdot\overline{a}\nabla\tilde{\overline{\rho}}+f(x+w(t),t)\;\;\textrm{with}\;\;\rho(\cdot,0)=g.\end{equation}
After defining $\overline{\rho}(x,t) = \tilde{\overline{\rho}}(x-w(t),t)$, it follows from It\^o's formula \cite{Kry2013} that $\overline{\rho}$ is a solution to the the SPDE
\[\partial_t\overline{\rho} = \nabla\cdot\overline{a}\nabla\overline{\rho}+\nabla\overline{\rho}\circ\Sigma \dd B_t+f \;\;\textrm{in}\;\;\R^d\times(0,\infty)\;\;\textrm{with}\;\;\overline{\rho}(\cdot,0)=g,\]
which is pathwise unique and therefore unique in law by Proposition~\ref{prop_wp_spde}.

A repetition of the argument of Lemma~\ref{lem_main} in the above setting proves with the Skorokhod representation theorem that there exist $\overline{X}\times\C([0,T];\R^d)$-valued random variables $(Y_k,W_k)$ on some probability space $(S,\mcS,\mathbf{P})$ with distribution equal to $(X^{\ve_k},w^{\ve_k})$.  The above considerations prove that the functions defined by the $Y_k$ converge in distribution on $L^2([0,T];H^1(B_n))\cap \C([0,T];L^2(\R^d))$ to a solution of the SPDE \eqref{final}.  We therefore conclude that, along this further subsequence $k\rightarrow\infty$, the $\rho^{\ve_k}$ converge in law to a solution of \eqref{final}.  Since the subsequence was arbitrary, this completes the proof.  \end{proof}

\appendix

\section{The stream matrix}\label{sec_stream}

In this section, we will recall the essentially well-known fact that an $L^2$-integrable, divergence-free drift $B$ admits a stationary stream matrix $S$ provided the dimension $d\geq 3$ and provided the correlations of $B$ decay faster than a square.  This is only a small modification of the analogous results in \cite{Koz1985}, \cite[Proposition~4, Proposition~5]{KozTot2017}, and \cite[Proposition~2.7]{Feh2022} which we include here to illustrate how it is that the conditional expectation appears.

\begin{prop} Assume~\eqref{steady}.  Let $d\geq 3$ and let $B\in L^2(\O;\R^d)$ satisfy that $\E[B]=0$, that $D\cdot B=0$ distributionally, and that, for some $c\in(0,\infty)$ and $\beta\in(2,\infty)$, for every $i,j\in\{1,\ldots,d\}$,
\begin{equation}\label{sm_1}\E\left[\abs{B_i(\tau_x \o)B_j(\tau_y\o)}\right]\leq c(\abs{x-y}\vee 1)^{-\beta}.\end{equation}
Then there exists skew-symmetric matrix $S=(S_{jk})_{j,k\in\{1,\ldots,d\}}\in L^2(\O;\R^{d\times d})$ that distributionally satisfies
\[D\cdot S = B-\underline{B}\;\;\textrm{on}\;\;\O.\]
\end{prop}

\begin{proof}  We will begin with a preliminary calculation.  Let $X=(X_i)_{i\in\{1,\ldots,d\}}\in L^\infty(\O;\R^d)$ be an arbitrary random variable satisfying \eqref{sm_1}, and for every $\a\in(0,1)$ let $S_\a\in\mathcal{H}^1(\O)$ denote the unique Lax-Milgram solution of the equation
\begin{equation}\label{sm_02}\a S_\a-D\cdot DS_\a=D\cdot X\;\;\textrm{in}\;\;\mathcal{H}^1(\O)\cap L^2(\O).\end{equation}
Since the boundedness of $X$ guarantees that we have the representation
\begin{align*}
S_\a(\o) & =\int_0^\infty \int_{\R^d}(4\pi s)^{-\frac{d}{2}}e^{-\a s-\frac{\abs{x}^2}{4s}}\frac{x}{2s}\cdot X(\tau_x\o)\dx\ds
\\ \nonumber  & =(4\pi)^{-\frac{d}{2}}\int_0^\infty \int_{\R^d}\abs{x}^{1-d}s^{-(\frac{d}{2}+1)}e^{-\a s\abs{x}^2-\frac{1}{4s}}\frac{x}{2\abs{x}}\cdot X(\tau_x\o)\dx\ds,
\end{align*}
it follows that there exists $c\in(0,\infty)$ depending on $d$ such that
\[\E[\abs{S_\a}^2]\leq c\int_{\R^d}\int_{\R^d}\abs{x}^{1-d}\abs{y}^{1-d}\E[\abs{X(\tau_x\o)}\abs{X(\tau_y)}]\dx\dy,\]
and then by \eqref{sm_1}, for some $c\in(0,\infty)$ depending on $d$,
\[\E[\abs{S_\a}^2]\leq c\int_{\R^d}\int_{\R^d}\abs{x}^{1-d}\abs{y}^{1-d}(\abs{x-y}\vee 1)^{-\beta}\dx\dy.\]
For every $x\in\R^d\setminus\{0\}$, we observe using $\beta\in(2,\infty)$ that, for some $c\in(0,\infty)$ depending on $d$,
\begin{align*}
& \int_{\R^d}\abs{y}^{1-d}(\abs{x-y}\vee 1)^{-\beta} =\int_{\{\abs{x-y}\leq 1\}}\abs{y}^{1-d}+\int_{\{\abs{x-y}>1\}}\abs{y}^{1-d}\abs{x-y}^{-\beta}
\\ & \leq c(\abs{x}\vee 1)^{1-d}+\int_{\{\abs{y}\geq 1\}}\abs{x-y}^{1-d}\abs{y}^{-\beta} \leq c(\abs{x}\vee 1)^{1-d}+\abs{x}^{1-d-\beta}\int_{\abs{y}\geq 1}\left(\frac{\abs{y}}{\abs{x}}\right)^{1-d}\abs{\frac{x}{\abs{x}}-\frac{y}{\abs{x}}}^{-\beta}
\\ & \leq c(1\vee \abs{x})^{1-d}+\abs{x}^{1-\beta}\sup_{\{\abs{x}=1\}}\int_{\R^d}\abs{y}^{1-d}\abs{x-y}^{-\beta}\leq c\left((\abs{x}\vee 1)^{1-d}+\abs{x}^{1-\beta}\right),
\end{align*}
where here we actually require only $\beta\in(1,\infty)$ to guarantee that the supremum is finite.  We then have that, for some $c\in(0,\infty)$ depending on $d$, for every $\a\in(0,1)$,
\[\E[\abs{S_\a}^2]\leq c\int_{\R^d}\abs{x}^{1-d}\left(\abs{x}^{1-d}+\abs{x}^{1-\beta}\right) \leq c,\]
which, after switching to radial coordinates, is finite and bounded uniformly in $\a\in(0,1)$ because $d\geq 3$ and $\beta\in(2,\infty)$.  It then follows from \eqref{sm_02} that, for some $c\in(0,\infty)$ depending on $d$ and \eqref{sm_1} but independent of $\a\in(0,1)$,
\begin{equation}\label{sm_3} \E[\abs{S_\a}^2+\E\left[\abs{DS_\a}^q\right]\leq c\left(1+\E\left[\abs{X}^2\right]\right).\end{equation}
After passing to a subsequence $\a\rightarrow 0$ there exists $S\in \mathcal{H}^1(\O)\cap L^2(\O)$ such that
\[S_\a\rightharpoonup S\;\;\textrm{weakly in}\;\;L^2(\O)\;\;\textrm{and}\;\;DS_\a\rightharpoonup DS\;\;\textrm{weakly in}\;\;L^2_{\textrm{pot}}(\O),\]
from which it follows that
\begin{equation}\label{sm_4}-D\cdot DS = D\cdot X\;\;\textrm{in}\;\;L^2_\textrm{pot}(\O).\end{equation}
The density of bounded functions in $L^2(\O;\R^d)$ proves that for every $X\in L^2(\O;\R^d)$ satisfying \eqref{sm_1} there exists a unique $S\in \mathcal{H}^1(\O)\cap L^2(\O)$ satisfying \eqref{sm_4} and the estimates of \eqref{sm_3}.

Now let $B=(B_i)_{i\in\{1,\ldots,d\}}\in L^2(\O;\R^d)$ for some $p\in[2,\infty)$ be mean zero and divergence-free in the sense that $\E[B]=0$ and $D\cdot B=0$, and let $B$ satisfy \eqref{sm_1}.  For every $j,k\in\{1,\ldots,d\}$ let $S_{jk}\in \mathcal{H}^1(\O)\cap L^2(\O)$ be the unique Lax-Milgram solution of the equation
\[-D\cdot DS_{jk}=D_jB_k-D_kB_j,\]
and let $S = (S_{jk})_{j,k\in\{1,\ldots,d\}}$.  We will now prove that $D\cdot S=B - \underline{B}$ in $L^2(\O;\R^d)$.  To see this, observe distributionally that, for every $i\in\{1,\ldots,d\}$, applying Einstein's summation convention,
\[D\cdot D(D\cdot S)_i = D_kD_k D_jS_{ij}  = D_j(D_jB_i-D_iB_j) = D\cdot D B_i.\]
That is, for every $i\in\{1,\ldots,d\}$, the difference $(D\cdot S_i-B_i)$ is spatially harmonic on $\O$.  It then follows that, for $s(x,t)=S(\tau_{x,t}\o)$ and $b(x,t) = B(\tau_{x,t}\o)$, for standard convolution kernels $\kappa^\ve$ of scale $\ve\in(0,1)$, we have from Proposition~\ref{prop_sublinear} that $((\nabla\cdot s)_i-b_i)*\kappa^\ve$ is $\P$-a.e.\ harmonic and sublinear on $\R^d$ for every $i\in\{1,\ldots,d\}$.  The zeroth order Liouville theorem therefore proves that $((\nabla\cdot s)_i-b_i)*\kappa^\ve$ is $\P$-a.e.\ constant on $\R^d$, and by the ergodic theorem we have $\P$-a.e.\ that
\[((D\cdot S)_i-B_i)*\kappa^\ve = \E\left[((D\cdot S)_i-B_i)*\kappa^\ve|\mcF_{\R^d}\right].\]
After passing to the limit $\ve\rightarrow 0$, we having using Proposition~\ref{prop_normalize} that
\[(D\cdot S)-B = \E[(D\cdot S)-B|\mathcal{F}_{\R^d}] = -\underline{B},\]
which completes the proof.  \end{proof}

\section{The convergence of the $w^\ve$ in law.}\label{sec_path}  In this section we provide general conditions under which the hypothesis \eqref{intro_weak_con} is satisfied.  We observe that the conditions below are always satisfied if $\underline{B}$ satisfies a finite range of dependence, or decorrelates quickly enough in time.

\begin{prop}  Assume~\eqref{steady}, let $T\in(0,\infty)$, and let $\underline{B}(\tau_{0,t}\o) = \underline{b}(t) = (\underline{b}_1,\ldots,\underline{b}_d)$.  For every $j\in \Z$ and $i\in\{1,\ldots,d\}$ let $x_{i,j} = \int_{j-1}^j\underline{b}_i(s)\ds$, let $x_j = (x_{1,j},\ldots,x_{d,j})$, and let $\F_0$ be the sigma algebra
\[\F_0 = \sigma(x_{i,j}\colon j\leq 0\;\;\;\textrm{and}\;\;i\in\{1,\ldots,d\}).\]
It follows from \eqref{steady} that $\E[x_{i,0}]=0$ and $\E[\abs{x_{i,0}}^2]<\infty$ for every $i\in\{1,\ldots,d\}$ and assume in addition that, for every $\theta\in\R^d$,
\[\sum_{j=1}^\infty\abs{\E\left[\theta\cdot x_j|\F_0\right]}\;\;\textrm{and}\;\; \langle\Sigma\Sigma^t\theta,\theta\rangle :=\E[\abs{\theta\cdot x_j}^2]+\sum_{j=2}^\infty \E\left[(\theta\cdot x_j)(\theta\cdot x_j)\right],\]
converge absolutely $\P$-a.e.\ for some positive definite matrix $\Sigma\in \R^{d\times d}$.  Finally, assume that for some $\d\in(0,1)$ we have that $\underline{B}\in L^{2+\d}(\O)$.  Then, the path $w^\ve(t) = \ve^{-1}\int_0^t\underline{b}(\nicefrac{s}{\ve^2})\ds$ satisfies that, as $\ve\rightarrow 0$,
\[ w^\ve \rightarrow \Sigma B_t\;\;\textrm{in law on $\C([0,T];\R^d)$},\]
for a standard $d$-dimensional Brownian motion $B$.  \end{prop}

\begin{proof}  For every $\theta\in\R^d$ and $n\in\N$ let $S^\theta(n) = \sum_{j=1}^n (X_j\cdot \theta)$.  It is then a consequence of \cite[Chapter~4, Theorem~19.1]{Bil1999} that, for every $\theta\in\R^d$ and $\ve\in(0,1)$, the processes
\[t\in[0,\infty)\mapsto \ve S(\lfloor \nicefrac{t}{\ve^2} \rfloor),\]
converges in distribution as $\ve\rightarrow 0$ to a one-dimensional Brownian motion with quadratic variation $\langle \Sigma\Sigma^t\theta,\theta\rangle$ in the Skorokhod space $\mathcal{D}([0,T])$.  It then follows from Prokhorov's theorem that the laws of the $(\ve S(\lfloor \nicefrac{t}{\ve^2} \rfloor))_{\ve\in(0,1)}$ are tight on $\mathcal{D}([0,T];\R^d)$, and it follows from Levy's characterization of Brownian motion that the processes converge in law to a standard $d$-dimensional Brownian motion with covariance matrix $\Sigma\Sigma^t$.  To conclude, we observe by definition that, for every $\eta\in(0,1)$,
\[\P\left(\sup_{t\in[0,T]}\abs{w^\ve(t)-\ve S(\lfloor \nicefrac{t}{\ve^2}\rfloor)}>\eta\right)\leq \P\left(\sup_{1\leq j\leq \lfloor \nicefrac{t}{\ve^2}\rfloor+1}\int_{j-1}^j\abs{b(s)}\ds\geq \sqrt{n}\eta\right).\]
It then follows that
\begin{align*}
& \limsup_{\ve\rightarrow 0}\P\left(\sup_{t\in[0,T]}\abs{w^\ve(t)-\ve S(\lfloor \nicefrac{t}{\ve^2}\rfloor)}>\eta\right)
\\ & \leq \P\left(\int_{n-1}^n\abs{b(s)}\ds\geq \sqrt{n}\eta\;\;\textrm{for infinitely may $n\in\N$}\right).
% y
\end{align*}
The assumptions, the time-stationarity, and Chebyshev's inequality prove that, for every $n\in\N$,
\[\P\left(\int_{n-1}^n\abs{b(s)}\ds\geq \sqrt{n}\eta\right)\leq n^{-\frac{2+\d}{2}}\eta^{-(2+\d)}\E\left[\int_{n-1}^n\abs{b(s)}^{2+\d}\ds\right] =  n^{-\frac{2+\d}{2}}\eta^{-(2+\d)}\E[\underline{B}^{2+\delta}].\]
It then follows from the Borel-Cantelli lemma and $\d\in(0,1)$ that
\[\P\left(\int_{n-1}^n\abs{b(s)}\ds\geq \sqrt{n}\eta\;\;\textrm{for infinitely may $n\in\N$}\right)=0,\]
which proves that, as $\ve\rightarrow 0$,
\begin{equation}\label{path_1} \abs{w^\ve-\ve S(\lfloor \nicefrac{t}{\ve^2}\rfloor)}\rightarrow 0\;\;\textrm{in probability in $\mathcal{D}([0,T])$.}\end{equation}
The distributional convergence of the $(\ve S(\lfloor \nicefrac{t}{\ve^2}\rfloor))_{\ve\in(0,1)}$ and the convergence in probability \eqref{path_1} imply that the paths $w^\ve$ converge in distribution in $\mathcal{D}([0,T])$ to $\Sigma B_t$.  The convergence in $\C([0,T];\R^d)$ then follows from the continuity of the paths, which completes the proof.  \end{proof}

\section*{Acknowledgements}

The author acknowledges financial support from the Engineering and Physical Sciences Research Council of the United Kingdom through the EPSRC Early Career Fellowship EP/V027824/1. 

\bibliographystyle{plain}
\bibliography{div_free}

\begin{thebibliography}{10}

\bibitem{AndBarDeuHam2014}
S.~Andres, M.T. Barlow, J.-D. Deuschel, and B.~Hambly.
\newblock Invariance principle for the random conductance model.
\newblock {\em Probab. Theory Relat. Fields}, 156:535--580, 2013.

\bibitem{Aub1963}
J.-P. Aubin.
\newblock Un th\'eor\`eme de compacit\'e.
\newblock {\em C. R. Acad. Sci. Paris}, 256:5042--5044, 1963.

\bibitem{AveMaj1991}
M.~Avellaneda and A.~J. Majda.
\newblock An integral representation and bounds on the effective diffusivity in
  passive advection by laminar and turbulent flows.
\newblock {\em Comm. Math. Phys.}, 138(2):339--391, 1991.

\bibitem{BenLioPap2011}
A.~Bensoussan, J.-L. Lions, and G.~Papanicolaou.
\newblock {\em Asymptotic analysis for periodic structures}.
\newblock AMS Chelsea Publishing, 2011.

\bibitem{Bil1999}
P.~Billingsley.
\newblock {\em Convergence of probability measures}.
\newblock Wiley Series in Probability and Statistics: Probability and
  Statistics. John Wiley \& Sons, Inc., New York, second edition, 1999.
\newblock A Wiley-Interscience Publication.

\bibitem{CanHauTon2021}
G.~Cannizzaro, L~Haunschmid-Sibitz, and F.~Toninelli.
\newblock $\sqrt{\log t}$-superdiffusivity for a brownian particle in the curl
  of the 2d gff.
\newblock {\em arXiv:2106.06264}, 2022.

\bibitem{Csa1973}
G.~T. Csanady.
\newblock {\em Turbulent diffusions in the environment}, volume~3.
\newblock Kluwer Academic, 1973.

\bibitem{DeuKos2008}
J.D. Deuschel and H.~K\"osters.
\newblock The quenched invariance principle for random walks in random
  environments admitting a bounded cycle representation.
\newblock {\em Ann. Inst. Henri Poincar\'{e} Probab. Stat. Statistiques},
  44(3):574--591, 2008.

\bibitem{Eva2010}
L.~C. Evans.
\newblock {\em Partial differential equations}, volume~19 of {\em Graduate
  Studies in Mathematics}.
\newblock American Mathematical Society, Providence, RI, second edition, 2010.

\bibitem{FanKom1997}
A.~Fannjiang and T.~Komorowski.
\newblock A martingale approach to homogenization of unbounded random flows.
\newblock {\em Ann. Probab.}, 25(4):1872--1894, 1997.

\bibitem{FanKom1999}
A.~Fannjiang and T.~Komorowski.
\newblock An invariance principle for diffusion in turbulence.
\newblock {\em Ann. Probab.}, 27(2):751--781, 1999.

\bibitem{FanKom2002}
A.~Fannjiang and T.~Komorowski.
\newblock Correction: ``{A}n invariance principle for diffusion in
  turbulence''.
\newblock {\em Ann. Probab.}, 30(1):480--482, 2002.

\bibitem{Feh2022}
B.~Fehrman.
\newblock Large-scale regularity in stochastic homogenization with
  divergence-free drift.
\newblock {\em arXiv:2006.16892}, 2022.

\bibitem{Fri1995}
U.~Frisch.
\newblock {\em Turbulence: The legacy of A.\ N.\ Kolmogorov}.
\newblock Cambridge University Press, 1995.

\bibitem{JikKozOle1994}
V.~V. Jikov, S.~M. Kozlov, and O.~A. Ole\u{\i}nik.
\newblock {\em Homogenization of differential operators and integral
  functionals}.
\newblock Springer-Verlag, 1994.

\bibitem{KipVar1986}
C.~Kipnis and S.R.S. Varadhan.
\newblock Central limit theorem for additive functionals of reversible markov
  processes with applications to simple exclusion.
\newblock {\em Commun. Math. Phys.}, 106:1--19, 1986.

\bibitem{KomOll2001}
T.~Komorowski and S.~Olla.
\newblock On homogenization of time-dependent random flows.
\newblock {\em Probab. Theory Related Fields}, 121(1):98--116, 2001.

\bibitem{KomOll2002}
T.~Komorowski and S.~Olla.
\newblock On the superdiffusive behaviour of passive tracer with a gaussian
  drift.
\newblock {\em J. Stat. Phys.}, 108:647--668, 2002.

\bibitem{Koz1985}
S.~M. Kozlov.
\newblock The averaging method and walks in inhomogeneous environments.
\newblock {\em Uspekhi Mat. Nauk}, 40(2(242)):61--120, 238, 1985.

\bibitem{KozTot2017}
G.~Kozma and B.~T\'{o}th.
\newblock Central limit theorem for random walks in doubly stochastic random
  environment: $\mathcal{H}_{-1}$ suffices.
\newblock {\em Ann. Probab.}, 45(6B):4307--4347, 2017.

\bibitem{Kry2013}
N.~V. Krylov.
\newblock A relatively short proof of {I}t\^{o}'s formula for {SPDE}s and its
  applications.
\newblock {\em Stoch. Partial Differ. Equ. Anal. Comput.}, 1(1):152--174, 2013.

\bibitem{LanOllYau1998}
C.~Landim, S.~Olla, and H.~T. Yau.
\newblock Convection-diffusion equation with space-time ergodic random flow.
\newblock {\em Probab. Theory Related Fields}, 112(2):203--220, 1998.

\bibitem{Lio1969}
J.-L. Lions.
\newblock {\em Quelques m\'ethodes de r\'esolution des probl\`emes aux limites
  non lin\'eaires}.
\newblock Dunod; Gauthier-Villars, Paris, 1969.

\bibitem{MonYag2007}
A.~S. Monin and A.~M. Yaglom.
\newblock {\em Statistical fluid mechanics: mechanics of turbulence. {V}ol.
  {I}}.
\newblock Dover Publications, 2007.

\bibitem{MonYag2007II}
A.~S. Monin and A.~M. Yaglom.
\newblock {\em Statistical fluid mechanics: mechanics of turbulence. {V}ol.
  {II}}.
\newblock Dover Publications, 2007.

\bibitem{Oel1988}
K.~Oelschl\"{a}ger.
\newblock Homogenization of a diffusion process in a divergence-free random
  field.
\newblock {\em Ann. Probab.}, 16(3):1084--1126, 1988.

\bibitem{Osa1983}
H.~Osada.
\newblock Homogenization of diffusion processes with random stationary
  coefficients.
\newblock In {\em Probability theory and mathematical statistics}, volume 1021
  of {\em Lecture Notes in Math.}, pages 507--517. Springer-Verlag, 1983.

\bibitem{PapVar1981}
G.~C. Papanicolaou and S.~R.~S. Varadhan.
\newblock Boundary value problems with rapidly oscillating random coefficients.
\newblock In {\em Random fields, {V}ol. {I}, {II}}, volume~27 of {\em Colloq.
  Math. Soc. J\'{a}nos Bolyai}, pages 835--873. North-Holland, Amsterdam-New
  York, 1981.

\bibitem{PapVar1982}
G.~C. Papanicolaou and S.~R.~S. Varadhan.
\newblock Diffusions with random coefficients.
\newblock In {\em Statistics and probability: essays in honor of {C}. {R}.
  {R}ao}, pages 547--552. North-Holland, Amsterdam, 1982.

\bibitem{SidSzn2004}
V.~Sidoravicius and A.-S. Sznitmal.
\newblock Quenched invariance principles for walks on clusters of percolation
  or among random conductances.
\newblock {\em Probab. Theory Relat. Fields}, 129:219--244, 2004.

\bibitem{Sim1987}
J.~Simon.
\newblock Compact sets in the space {$L^p(0,T;B)$}.
\newblock {\em Ann. Mat. Pura Appl. (4)}, 146:65--96, 1987.

\bibitem{Tot2018}
B.~T\'oth.
\newblock Quenched central limit theorem for random walks in doubly stochastic
  random environment.
\newblock {\em Ann. Probab.}, 46(6):3558--3577, 2018.

\bibitem{TotVal2012}
B.~T\'oth and B.~V\'alko.
\newblock Superdiffusive bounds on self-repellent brownian polymers and
  diffusion in the curl of the gaussian free field in $d=2$.
\newblock {\em J. Stat. Phys.}, 147:113--131, 2012.

\end{thebibliography}

\end{document}